\documentclass[reqno,11pt, a4paper]{amsart}

\usepackage[T1]{fontenc} 

\usepackage{amssymb}
\usepackage{amsmath}
\usepackage{amsfonts}
\usepackage{amsthm}
\usepackage{mathtools}
\usepackage{mathabx}
\usepackage{fullpage}
\usepackage{comment}
\usepackage[dvipsnames]{xcolor}
\usepackage{graphicx}
\usepackage{float}

\usepackage{cite}

\usepackage{enumitem} 

\usepackage[export]{adjustbox}
\usepackage{bbold}
\usepackage[percent]{overpic}

\usepackage{booktabs,array}
\usepackage{url} 
\usepackage{hyperref}

\usepackage{pgfplots}
\pgfplotsset{compat=1.6}

\usepackage{mdframed}

\usepackage{caption}
\usepackage{subcaption}

\theoremstyle{plain}%
\newtheorem{theorem}{Theorem}[section]
\newtheorem{lemma}[theorem]{Lemma}
\newtheorem{proposition}[theorem]{Proposition}

\newtheorem*{conjecture*}{Conjecture}

 \numberwithin{equation}{section}

\theoremstyle{definition}
\newtheorem{definition}[theorem]{Definition}

\theoremstyle{remark}
\newtheorem{remark}[theorem]{Remark}

\let \leq \leqslant
\let \geq \geqslant

\DeclareMathOperator{\tr}{tr}
\DeclareMathOperator{\sgn}{sgn} 
\DeclareMathOperator{\Cov}{Cov} 
\DeclareMathOperator{\support}{supp} 
 
\DeclareMathOperator{\image}{Image} 
\DeclareMathOperator{\dist}{dist}

\title{Decay of correlations in finite Abelian lattice gauge theories}

\author{Malin P. Forsstr\"om}
\address[Malin P. Forsstr\"om]{Department of Mathematics, KTH Royal Institute of Technology, 100 44 Stockholm, Sweden.}
\email{malinpf@kth.se}

\begin{document}

\begin{abstract}
    In this paper, we study lattice gauge theory on \( \mathbb{Z}^4 \) with finite Abelian structure group. When the inverse coupling strength is sufficiently large, we give an upper bound on the decay of correlations of local functions and compute the leading order term for both the expected value of the spin at a given plaquette as well as for the two-point correlation function. Moreover, we give an upper bound on the dependency of the size of the box on which the model is defined. The results in this paper extend and refine results by Chatterjee and Borgs.
\end{abstract}

 \maketitle

\section{Introduction}

\subsection{Background}
 
Gauge theories are crucial tools in modern physics. For instance, they are used to formulate the Standard Model.
These quantum field theories describe how different types of elementary particles interact. Even though such models have proved to be very successful in physics, they are not mathematically well-defined. This problem was considered important enough to be chosen as one of the Millenium Problems by the Clay Mathematics Institute~\cite{jw}.

Euclidean lattice gauge theories, with underlying structure group given by e.g. \( U(1)\), \( SU(2) \) or \( SU(3) \), appear as natural and well-defined discretizations of gauge theories on hyper-cubic lattices~\cite{wilson}. These discrete models have been proven to be very useful as tools to study the corresponding quantum field theories using e.g.\ simulations, high temperature expansions and low temperature expansions \cite{k1979}. However, there is also hope that one would be able to take a scaling limit and in this way obtain a rigorously defined continuum gauge theory.
As a first step in this direction, it is often instructive to try to understand relevant properties of slightly simpler models of the same type. 
The decay of correlations is an important property to try to understand in any model in statistical physics. For lattice gauge theories, this type of property is given further relevance due its connection with the mass gap problem in Yang-Mills theories. 
This is the main motivation for the current paper, where we study the decay of correlations in Abelian lattice gauge theories. 

While lattice gauge theories with a finite Abelian structure group are not of known direct physical significance in the context of the Standard Model, they provide toy models for development of tools and ideas which can later be generalized to more physically relevant models. For this reason, they have been studied in the physics literature, see, e.g., \cite{frolich-spencer82, b1984, mp1979} and the references therein, as well as in the mathematical literature, see e.g.\ \cite{c2018b, flv2020, b1984}.

\subsection{Lattice gauge theories with Wilson action}

The lattice $\mathbb{Z}^4$ has a vertex at each point in $\mathbb{R}^4$ with integer coordinates, and an edge between nearest neighbors, oriented in the positive direction, so that there are exactly four positively oriented edges emerging from each vertex \( x \), denoted by \( dx_i , \,i=1,\ldots,4 \). We will let \( -dx_i \) denote the edge with the same end points as \( dx_i \) but with opposite orientation.
Each pair \( dx_i \) and \( dx_j \) of directed edges defines an \emph{oriented plaquette} \( dx_i \wedge dx_j \). If \( i<j \), we say that the plaquette \( dx_i \wedge dx_j \) is positively oriented, and if \( i >j \), we say that the plaquette   \( dx_i \wedge dx_j = - dx_j \wedge dx_i \) is negatively oriented. 

Given real numbers \( a_1<b_1 \), \( a_2<b_2 \), \( a_3<b_3 \) and \( a_4<b_4 \), we say that \( B = \bigl( [a_1,b_1] \times [a_2,b_2] \times [a_3,b_3] \times [a_4,b_4] \bigr) \cap \mathbb{Z}^4 \) is a box. When \( B \) is a box, we write $E_B$ for the set of (positively and negatively) oriented edges both of whose endpoints are  contained in $B$, and $P_B$ for set of the oriented plaquettes whose edges are all contained in $E_B$. We will often write $e$ and $p$ for elements of $E_B$ and $P_B$, respectively. 

In this paper we will always assume that a finite and Abelian group \( G \) has been given. This group will be referred to as the \emph{structure group}.
We let $\Sigma_{E_B}$ be the set of \( G \)-valued 1-forms on  $E_B$, i.e., the set of functions \( \sigma \colon E_B \to G \) with the property that \( \sigma(e) = -\sigma(-e) \). Whenever \( \sigma \in \Sigma_{E_B} \) and \( e \in E_B \), we write \( \sigma_e \coloneqq \sigma(e) \).
Each element $\sigma  \in \Sigma_{E_B}$ induces a configuration \( d\sigma \) on \( P_B \) by assigning
\begin{equation}\label{p-e}
(d\sigma)_p:= \sigma_{e_1}+\sigma_{e_2}+\sigma_{e_3} +\sigma_{e_4}, \quad p \in  P_B,
\end{equation}
where $e_1,e_2, e_3, e_4$ are the edges in the boundary \( \partial p\) of $ p$, directed according to the orientation of the plaquette \( p \), see Section~\ref{sec: plaquettes}. The set of all configurations on \( P_B \) which arise in this way will be denoted by \( \Sigma_{P_B} \).

Next, we let \( \rho \) be a faithful, irreducible and unitary representation of $G$. 
With \( G \) and \(\rho\) fixed, we define the \emph{Wilson action} by
\begin{equation}\label{eq: Wilson action action}
    S(\sigma) 
    \coloneqq - \sum_{p \in P_B}  \Re  \tr \rho\bigl((d\sigma)_p\bigl), \quad \sigma \in \Sigma_{E_B}.
\end{equation} 
Letting \( \mu_H \) denote the uniform measure on \( \Sigma_{E_B}\) and fixing some \( \beta \geq 0 \), we obtain an associated probability measure $\mu_{B,\beta}$ on $\Sigma_{E_B}$ by weighting \( \mu_H \) by the Wilson action:
\begin{equation}\label{eq: Wilson action}
    \mu_{B,\beta}(\sigma) \coloneqq Z^{-1}_{B,\beta} e^{-\beta S(\sigma)} \,  \mu_H(\sigma), \quad \sigma \in \Sigma_{E_B},
\end{equation} 
where \( Z_{B,\beta} \) is a constant that ensures that \( \mu_{B,\beta} \) is a probability measure. The probability measure \( \mu_{B,\beta} \) describes lattice gauge theory on \( B \) with structure group \( G \), representation \( \rho \), coupling parameter \( \beta \) and \emph{free boundary conditions}.
We let \( \mathbb{E}_{B,\beta} \) denote expectation with respect to \( \mu_{B,\beta} \).

\subsection{A distance between sets}
Let \( B \) and \( B' \) be two boxes in \( \mathbb{Z}^4 \) with \( B' \subseteq B \).
In all of our results, we need a measure of the distance between sets of plaquettes \( P_1,P_2 \subseteq P_B \). To be able to define such a measure, we now introduce the following graph.
Given \( \omega \in \Sigma_{P_B} \) and \( \omega' \in \Sigma_{P_{B'}} \), let \( \mathcal{G}(\omega,\omega') \) be the graph with vertex set \(  \support \omega  \cup \support \omega' \), and an edge between two distinct plaquettes \( p_1,p_2 \in \support \omega \cup \support \omega'\) if \( p_1 \) and \( \pm p_2 \) are both in the boundary of some common 3-cell (see also Definition~\ref{definition: graph definition}). 
For distinct \( p_1,p_2 \in P_B \), let
\begin{equation}\label{eq: definition of distance between points}
    \begin{split}
        \dist_{B,B'}(p_1,p_2) \coloneqq  \frac{1}{2} \min \Bigl\{ |\support \omega| + |\support \omega'| \colon \omega\in \Sigma_{P_B}, \, \omega'\in \Sigma_{P_{B'}} \text{ s.t. } p_1 \text{ and } p_2 & \\\text{are in the same  connected component of } \mathcal{G}\bigl(\omega,\omega'\bigr) &\Bigr\},
    \end{split}
\end{equation}
for \( p \in P_B \), let \( \dist_{B,B'}(p,p) \coloneqq 0 \),
and for sets \( P_1,P_2 \subseteq P_B \),  let
\begin{equation}\label{eq: definition of distance}
    \begin{split}
        \dist_{B,B'}(P_1,P_2) \coloneqq \min_{p_1 \in P_1,\, p_2 \in P_2} \dist_{B,B'}(p_1,p_2).
    \end{split}
\end{equation} 
When \( B' = B \), we write \( \dist_{B} \) instead of \( \dist_{B,B} \).
We mention that for any two distinct plaquettes \( p_1 \) and \( p_2 \), one can show that \( \dist_{B,B'}(p_1,p_2) \) is bounded from above and below by some constant times the graph distance (in the lattice \( \mathbb{Z}^4 \)) between the corners of \( p_1 \) and the corners of \( p_2 \).

\subsection{Preliminary notation}
To simplify notation, for \( \beta \geq 0 \) and \( g \in G \), we let 
\begin{equation}\label{eq: phi def}
    \phi_\beta(g) \coloneqq \frac{e^{\beta \Re \tr \rho(g)}}{e^{\beta \Re \tr \rho(0)}},
\end{equation} 
and
\begin{equation*}
    \alpha(\beta) \coloneqq   \sum_{g \in G\smallsetminus \{ 0 \}} \phi_\beta(g)^2.
\end{equation*} 
The function \( \alpha(\beta) \) will be used to express upper bounds on error terms in our main results. We mention that for any finite Abelian group \( G \) with a faithful representation \( \rho \), there are constants \( C> 0\) and \( \xi>0 \) such that \( \alpha(\beta) \leq C e^{- \beta \xi} \). In other words, \( \alpha(\beta) \) decays exponentially in \( \beta \).

Next, for \( \beta \geq 0 \) such that \(  30 \alpha(\beta) < 1 \), we define
     \begin{equation}\label{eq: C1}
        C_1(\beta) \coloneqq   \frac{20 }{15^2 (1 - 5\alpha(\beta))^2} \biggl[
            1
            + 
            \frac{2 }{1-30 \alpha(\beta)  }
            \biggr] ,
    \end{equation}
    and
    \begin{equation}\label{eq: C2}
        C_2 \coloneqq  30.
    \end{equation}  
    We note that as \( \beta \rightarrow \infty \), \( C_1(\beta) \searrow 4/9 \).
    
    When \( P \) is a set of plaquettes, we let \( \delta P \) denote the set of all plaquettes in \( P \) which shares a 3-cell with some plaquette which does not belong to \( P \).

\subsection{Main results}
 
In several recent papers, the expected value of  Wilson loop observables have been rigorously analyzed with probabilistic techniques, for different structure groups~\cite{c2018b,cao2020,flv2020,gs2021}. The Wilson loop is an important observable in lattice gauge theories because it is believed to be related to the energy required to separate a pair of quarks~\cite{wilson,c2016}. Another important observable is the spin-spin-correlation function, which is thought to be related to the so-called mass gap of the model~\cite{c2016}. In the first three main results of this paper, we study variants of this function in the low-temperature regime, by giving results which describe the decay of correlations of local functions.  The first of these results is the following theorem. To give the statement, when \( \omega \in \Sigma_{P_B} \) and \( P \subseteq P_B \) and \( P = -P \), we let \( \omega|_{P} \coloneqq (\omega_p \cdot \mathbb{1}_{p \in P})\) denote the restriction of \( \omega \) to \( P \) in the natural way (see also Section~\ref{section: restrictions}).

\begin{theorem}\label{theorem: correlation length}
    Let \( B \) be a box in \( \mathbb{Z}^4 \), and let  \( \beta \geq 0 \) be such that \(  30 \alpha(\beta) < 1 \). Further, let \( f_1,f_2 \colon \Sigma_{P_B} \to \mathbb{C} \) and assume that there are disjoint sets \( P_1,P_2 \subseteq P_B \) such that for all \( \omega \in \Sigma_{P_B} \) we have \( f_1(\omega ) = f_1(\omega|_{P_1})\) and \( f_2(\omega) = f_2(\omega|_{P_2}) \). Then, if \( \sigma \sim \mu_{B,\beta} \), we have
    \begin{equation}\label{eq: correlation length ii}
        \Bigl| \Cov \bigl( f_1( d\sigma),f_2 (d\sigma) \bigr)\Bigr| \leq  C_1  \| f_1\|_\infty \, \|f_2 \|_\infty \bigl( C_2 \alpha(\beta) \bigr)^{ \dist_{B}(P_1,P_2) } ,
    \end{equation} 
    where \( C_1 = C_1(\beta)  \)  and \( C_2 \) are given by in~\eqref{eq: C1} and~\eqref{eq: C2} respectively.
    In particular, if \( p_1 \in P_B \) and \( p_2 \in P_B \) are distinct, then 
    \begin{equation}\label{eq: correlation length}
        \Bigl| \Cov \bigl( \tr \rho( (d\sigma)_{p_1}),\tr \rho ((d\sigma)_{p_2}) \bigr)\Bigr| \leq  C_1 (\dim \rho )^2\bigl( C_2 \alpha(\beta) \bigr)^{ \dist_{B}( p_1 , p_2 ) } .
    \end{equation}  
\end{theorem}

\begin{remark}
    Since \( \alpha(\beta) \) decays exponentially in \( \beta \), Theorem~\ref{theorem: correlation length} shows that the covariance of two local functions decays exponentially both in \( \beta \) and in the distance between the corresponding sets \( P_1 \) and \( P_2 \).
\end{remark}

\begin{remark}
    We mention that if one for disjoint plaquettes \( p_1,p_2 \in P_B \) knew which configurations attained the minimum in~\eqref{eq: definition of distance between points}, then the methods used in this paper could be adapted slightly to describe the first order behaviour of \( \Cov\bigl(\tr \rho((d\sigma)_{p_1}),  \tr \rho((d\sigma)_{p_2}) \bigr) \) for sufficiently large \( \beta \).
\end{remark}

With Theorem~\ref{theorem: correlation length} at hand, it is natural to ask how the decay of the covariance in~\eqref{eq: correlation length} relates to the so-called spin-spin correlation \( \mathbb{E}_{B,\beta}\bigl[\tr \rho\bigl((d\sigma)_p\bigr) \tr\rho\bigl((d\sigma)_{p'}\bigr)\bigr] \). The next theorem, which improves upon a special case of Lemma~4.2 in~\cite{b1984}, is a first step towards answering this question.

\begin{theorem}\label{theorem: typical spin at plaquette}
    Let \( B \) be a box in \( \mathbb{Z}^4 \), and let \( \beta \geq 0 \) be such that \(  5 \alpha(\beta) < 1 \).
    Further, let \( p \in P_B \) and \( f \colon G \to \mathbb{C} \). Then, if \( \dist_{B}\bigl(\{ p \}, \delta P_B\bigr) >11 \), we have
    \begin{equation}\label{eq: spin expectation}
        \begin{split}
            &\biggl| \mathbb{E}_{B,\beta}\bigl[ f\bigl((d\sigma)_p\bigr)\bigr] - \Bigl(  f(0) +
            \sum_{e \in \partial p} \sum_{g \in G} 
            \bigl(f(g)-f(0) \bigr) \,  \phi_\beta(g)^{12} \Bigr) \biggr|
            \leq
            \frac{\bigl(5\alpha(\beta)\bigr)^{11}}{1 - 5\alpha(\beta) }  
            \, \max_{g \in G} \bigl|f(g) -f(0)\bigr|.
        \end{split}
    \end{equation} 
\end{theorem}

\begin{remark}
    When \( \beta \) tends to infinity, we have \( \sum_{e \in \partial p} \sum_{g \in G}  \phi_\beta(g)^{12} \asymp \alpha(\beta)^6 \), and hence the right hand side of~\eqref{eq: spin expectation} will in general tend to zero much faster than the term \( \sum_{e \in \partial p} \sum_{g \in G} \bigl(f(g)-f(0) \bigr) \,  \phi_\beta(g)^{12} \) on the left hand side of the same equation. Consequently, Theorem~\ref{theorem: typical spin at plaquette} captures the first- and second-order behaviour of \( \mathbb{E}_{B,\beta}\bigl[ f\bigl((d\sigma)_p\bigr)\bigr] \).
\end{remark}

Combining Theorem~\ref{theorem: correlation length} and Theorem~\ref{theorem: typical spin at plaquette}, we obtain the following result as a corollary.
 
\begin{theorem}\label{theorem: correlation length ii}
    Let \( B \) be a box in \( \mathbb{Z}^4 \), and let \( \beta \geq 0 \) be such that \(  30  \alpha(\beta) < 1 \). Further, let \( p_1,p_2 \in P_B \) be distinct, and let \( f_1,f_2 \colon G \to \mathbb{C} \). Then, if \( \dist_{B}\bigl(\{ p_1,p_2 \}, \delta P_B\bigr) >11 \), we have
    \begin{equation} \label{eq: eq in last theorem}
        \begin{split}
            &\biggl| \mathbb{E}_{B,\beta}\bigl[f_1\bigl((d\sigma)_{p_1}\bigr)f_2\bigl((d\sigma)_{p_2}\bigr)\bigr] - \prod_{j\in\{1,2 \}}\Bigl( f_j(0) + \sum_{e \in \partial p}  \sum_{g \in G} \bigl(f_j(g)-f_j(0) \bigr) \, \phi_\beta(g)^{12}  \Bigr) \biggr|
            \\&\qquad\leq
            C_1 \| f_1 \|_\infty \| f_2\|_\infty \bigl( C_2 \alpha(\beta)\bigr)^{ \dist_{B}(\{ p_1 \},\{ p_2 \}) } 
            +  
            8 \|f_1\|_\infty    \|f_{2}\|_\infty  \frac{ (5 \alpha(\beta))^{11}}{1 - 5\alpha(\beta) } .
        \end{split}
    \end{equation}
    where \( C_1 = C_1(\beta)  \) and \( C_2 \),  are given by~\eqref{eq: C1}, and~\eqref{eq: C2} respectively.
\end{theorem}

\begin{remark}
For \( N \geq 1 \), let  \( B_N \) be the box \( [-N,N]^4 \cap \mathbb{Z}^4\).
By applying Ginibre's inequality~\cite{g1970}, one can show that whenever \( f \) is a real-valued function which depends only on a finite number of plaquettes, then the limit
\begin{align}\label{wilsonlimit}
\lim_{N \to \infty} \mathbb{E}_{B_N,\beta} \bigl[ f(d\sigma)\bigr]
\end{align}
exists and is translation invariant (see e.g. Section~2.3 in~\cite{flv2020}). From this result, it follows that Theorem~\ref{theorem: correlation length}, Theorem~\ref{theorem: typical spin at plaquette} and Theorem~\ref{theorem: correlation length ii} holds also in this limit.
\end{remark}

By using the same strategy as for the proof of Theorem~\ref{theorem: correlation length}, we obtain the following result, which extends Theorem~5.3~in~\cite{c2018b}. In this result, \( \dist_{TV}(X,Y) \) denotes the total variation distance between two random variables \( X \) and \( Y \).

\begin{theorem}[Compare with Theorem~5.3 in~\cite{c2018b} and Theorem~2.4 in~\cite{b1984}]\label{theorem: chatterjees correlation result}
    Let \( B \) and \( B' \) be two boxes in \( \mathbb{Z}^4 \) with \( B' \subsetneq B \), and let \( \beta \geq 0 \) be such that \(  30 \alpha(\beta) < 1 \).
    Further, let \( P\subseteq P_{B'} \), and let \( \sigma \sim \mu_{B,\beta} \) and \( \sigma'\sim \mu_{B',\beta} \).
    Then
    \begin{equation*}
        \begin{split}
            &
            \dist_{TV}\bigl((d\sigma)|_{P}, (d\sigma')|_{P}\bigr) 
            \leq  
            C_1  |P| \bigl( C_2 \alpha(\beta) \bigr)^{ \dist_{B,B'}(P,P_B\smallsetminus P_{B'}) },
        \end{split}
    \end{equation*}  
    where \( C_1 = C_1(\beta)  \) and \( C_2 \) are given by~\eqref{eq: C1}  and~\eqref{eq: C2} respectively. 
\end{theorem}

\begin{remark}
In contrast to Theorem~5.3 in~\cite{c2018b}, which hold only for \( G = \mathbb{Z}_2 \), Theorem~\ref{theorem: chatterjees correlation result} is valid for any finite Abelian group. Moreover, our proof can easily be adapted to work for other lattices such as \( \mathbb{Z}^n \) for \( n \geq 3 \), as well as for other actions such as the Villain action.
Moreover, we mention that even in the case of \( {G} = \mathbb{Z}_2 \), we use a completely different proof strategy than the strategy used in the corresponding proof in~\cite{c2018b}.
\end{remark}

\begin{remark}
    By~\cite{bdi1975}, a critical value for \( \beta \) in the case \( G = \mathbb{Z}_2 \) is given by \( 0.22 \). In comparison,  when \( G = \mathbb{Z}_2 \), the assumption on \( \beta \) in the above results is either that \( 5 e^{-4\beta} < 1\) (equivalently, \( \beta \geq 0.40 \)) or that \( 30 e^{-4\beta} < 1\) (equivalently, \( \beta \geq 0.8 \)).
\end{remark}

\begin{remark}
    In this paper, we always use the measure given by~\eqref{eq: Wilson action}, corresponding to free boundary conditions. However, with minor changes to the proofs in the paper, one can obtain results analogous to our main results for zero or periodic boundary conditions.
\end{remark}

\begin{remark}
    Using the theory as outlined in~\cite{cao2020}, some of the ideas in this paper might extend to finite non-Abelian structure groups as well. However, in some of the proofs, we use tools from discrete exterior calculus which are not valid in a non-Abelian setting. Consequently,  such a generalization would be non-trivial. 
\end{remark}
 
\begin{remark}
    The main novelty of this paper is the use of a coupling argument to obtain upper bounds on both the covariance of local functions and on the total variation distance. Also, we give natural extensions and generalizations of several technical but useful lemmas from~\cite{flv2020}.
\end{remark}

\subsection{Structure of the paper}

In Section~\ref{sec: preliminaries}, we give a brief summary of the foundations of discrete external calculus on hypercubic lattices, which we will use throughout the rest of this paper.
Next, in Section~\ref{sec: frustrated sets}, we give upper and lower bounds for the probability that certain plaquette configurations arise in spin configurations. 
In Section~\ref{section: paths}, we discuss the structure and properties of paths in the graph \( \mathcal{G}(\omega^{(0)},\omega^{(1)}) \) for \( \omega^{(0)},\omega^{(1)} \in \Sigma_{P_B} \).
In Section~\ref{sec: connected sets}, we define a notion of two sets being connected, and prove a lemma suggesting the usefulness of this concept.
In Section~\ref{sec: distances}, we define another measure \( \dist_B^*(p_1,p_2) \) of the distance between two plaquettes, and state and prove a lemma which gives a relationship between this function and the function \( \dist_{B,B'}(p_1,p_2) \) defined in the introduction.
These results are then used in Section~\ref{sec: probability of being connected} to give an upper bound on the probability that two sets are connected.
In Section~\ref{sec: covariance}, we give a connection between the event that two sets are connected and the covariance of local functions supported on these sets, and then use this connection to give a proof of Theorem~\ref{theorem: correlation length}.
In Section~\ref{sec: total variation}, we give a similar connection between the event that two sets are connected and the total variation distance, and then use this observation to give a proof of Theorem~\ref{theorem: chatterjees correlation result}.
Finally, in Section~\ref{sec: expected spin}, we prove Theorem~\ref{theorem: typical spin at plaquette} and Theorem~\ref{theorem: correlation length ii}.

\section{Preliminaries}\label{sec: preliminaries}

\subsection{Discrete exterior calculus}

In this section, we give a very brief overview of discrete exterior calculus on the cell complexes of \( \mathbb{Z}^n \) for \( n \in \mathbb{N} \). For a more thorough background on discrete exterior calculus, we refer the reader to~\cite{c2018b}. 

All of the results in this section are obtained under the assumption that an Abelian group \( G \), which is not necessarily finite, has been given. In particular, they all hold for \( G=\mathbb{Z} \).

\subsubsection{Oriented edges (1-cells)}

The graph $\mathbb{Z}^n$ has a vertex at each point \( x \in \mathbb{Z}^n \) with integer coordinates and an (undirected) edge between nearest neighbors. We associate to each undirected edge \( \bar e \) in \( \mathbb{Z}^n \) exactly two directed or \emph{oriented} edges \( e \) and \( -e \) with the same endpoints as \( \bar e \); $e$ is directed so that the coordinate increases when traversing $e$, and $-e$ is directed in the opposite way.

Let \( \mathbf{e}_1 \coloneqq (1,0,0,\ldots,0)\), \( \mathbf{e}_2 \coloneqq (0,1,0,\ldots, 0) \), \ldots, \( \mathbf{e}_n \coloneqq (0,\ldots,0,1) \) and let \( d{\mathbf{e}}_1 \), \ldots, \( d{\mathbf{e}}_n \)  denote the \( n \) oriented edges with one endpoint at the origin which naturally correspond to these unit vectors (oriented away from the origin). We say that an oriented edge \( e \) is \emph{positively oriented} if it is equal to a translation of one of these unit vectors, i.e., if there exists a point \( x \in \mathbb{Z}^n \) and an index \( j \in \{ 1,2, \ldots, n\} \) such that \( e = x + d{\mathbf{e}}_j \). If \( x \in \mathbb{Z}^n \) and  \( j \in \{ 1,2, \ldots, n\} \), then we let \( dx_j \coloneqq x + d{\mathbf{e}}_j  \).

Given a box \( B \), we let \( E_B \) denote the set of oriented edges whose end-points are both in \( B \), and let \( E_B^+ \) denote the set of positively oriented edges in \( E_B \).

\subsubsection{Oriented \( k \)-cells}

For any two oriented edges \( e_1 \in E_B \) and \( e_2\in E_B \), we consider the wedge product \( e_1 \wedge e_2  \) satisfying \( e_1 \wedge e_1 = 0 \) and
\begin{equation}\label{eq: reversing wedge order}
     e_1 \wedge e_2 = -(e_2 \wedge e_1) = (-e_2) \wedge e_1 = e_2 \wedge (-e_1).
\end{equation}
If \( e_1 \), \( e_2 \), \ldots, \( e_k \) are oriented edges which  do not share a common endpoint, we  set \( e_1 \wedge e_2  \wedge \cdots \wedge e_k = 0 \).

If \( e_1 \), \( e_2 \), \ldots, \( e_k \) are oriented edges and \( e_1 \wedge \cdots \wedge e_k \neq 0 \), we say that \( e_1 \wedge \cdots  \wedge e_k \) is an oriented \( k \)-cell. If there exists an \( x \in \mathbb{Z}^n \) and \( j_1 <j_2 <\cdots< j_k \) such that \( e_i = d{x}_{j_i} \), then we say that \( e_1 \wedge \cdots \wedge e_k \) is \emph{positively oriented} and that \( -( e_1 \wedge \cdots \wedge e_k) \) is \emph{negatively oriented}. Using~\eqref{eq: reversing wedge order}, this defines an orientation for all \( k \)-cells. If not stated otherwise, we will always consider $k$-cells as being oriented.

If \( A \) is a set of oriented \( k \)-cells, we let \( A^+ \) denote the set of positively oriented \( k \)-cells in \( A \).  If \( A = -A \) , then we say that \( A \) is \emph{symmetric}.

\subsubsection{Oriented plaquettes} \label{sec: plaquettes}
We will usually say \emph{oriented plaquette} instead of oriented \(2\)-cell.
If \( x \in \mathbb{Z}^n \) and \( 1 \leq j_1 <   j_2 \leq  \), then  \( p \coloneqq dx_{j_1} \wedge d{{x}}_{j_2} \) is a positively oriented plaquette, and we define 
\begin{equation*}
    \partial p \coloneqq \{ 
    dx_{j_1},  
    (d(x +  \mathbf{e}_{j_1}))_{j_2}, 
    -(d(x + \mathbf{e}_{j_2}))_{j_1}, 
    - dx_{j_2}
    \}.
\end{equation*} 
If \( e \) is an oriented edge, we let \( \hat \partial e \) denote the set of oriented plaquettes \( p \) such that \( e \in \partial p \).

We let \( P_B \) denote the set of oriented plaquettes whose edges are all in \( E_B \).

\subsubsection{Discrete differential forms}

A \( G \)-valued function \( f \) defined on a subset of the set of \( k \)-cells in \( \mathbb{Z}^n \) with the property that \( f(c) = -f(-c) \) is called a \emph{\( k \)-form}. 
If \( f \) is a \( k \)-form which takes the value \( f_{j_1,\ldots, j_k}(x) \) on \( dx_{j_1} \wedge \cdots \wedge dx_{j_k} \), it is useful to represent its values on the $k$-cells at $x \in \mathbb{Z}^n$ by the formal expression
\begin{equation*}
    f(x) = \sum_{1 \leq j_1 < \cdots < j_k\leq n} f_{ j_1,\ldots, j_k}(x) \,  dx_{j_1} \wedge \cdots \wedge dx_{j_k}.
\end{equation*}
To simplify notation, if \( c \coloneqq  dx_{j_1} \wedge \cdots \wedge dx_{j_k} \) is a \( k \)-cell and \( f \) is a \( k \)-form we often write \( f_c \) instead of \( f_{j_1,\ldots, j_k}(x) \).

Given a \( k \)-form \( f \), we let \( \support f \) denote the support of \( f \), i.e.\ the set of all oriented \( k \)-cells \( c \) such that \( f(c) \neq 0 \).

Now let \( B \) be a box, and recall that for \( k \in \{ 1,2,\ldots, n \} \), a \( k \)-cell \( c \) is said to be in \( B \) if all its corners are in \( B \). The set of \( G \)-valued \( k \)-forms with support in the set of \( k \)-cells that are contained in a \( B \) will be denoted by \( \Sigma_{B,k} \).
The set of \( G \)-valued 1-forms with support in \( E_B \) will also be denoted by \( \Sigma_{E_B} =\Sigma_{B,k}\), and will referred to as \emph{spin configurations}, and the set of \( G \)-valued 2-forms with support in \( \Sigma_{P_B} \) will be referred to as \emph{plaquette configurations}.

\subsubsection{The exterior derivative}
Given \( h \colon \mathbb{Z}^n \to G \), \( x \in \mathbb{Z}^n \), and \( i \in \{1,2, \ldots, n \} \), we let 
\begin{equation*}
    \partial_i h(x) \coloneqq h(x+\mathbf{e}_i) - h(x) .
\end{equation*}
If \( k \in \{ 0,1,2, \ldots, n-1 \} \) and \( f \) is a \( G \)-valued \( k \)-form, we define the \( (k+1) \)-form \( df \) via the formal expression
\begin{equation*}
    df(x) = \sum_{1 \leq j_1 < \cdots < j_k\leq n} \sum_{i=1}^n \partial_i f_{j_1,\ldots, j_k} (x) \, dx_i \wedge (dx_{j_1} \wedge \cdots \wedge dx_{j_k}), \quad x \in \mathbb{Z}^n.
\end{equation*}
The operator \( d \) is called the \emph{exterior derivative.}

\subsubsection{Boundary operators}\label{sec: boundary}

If \( \hat x  \in \mathbb{Z}^n \) and \( j \in \{ 1,2, \ldots, n \} \), then 
\begin{equation*}\label{equation: simplest differential form}
\begin{split}
    &d(\mathbb{1}_{x = \hat x} \, dx_j) 
    = \sum_{i=1}^n (\partial_i \mathbb{1}_{x = \hat x} )\,  dx_i \wedge dx_j
    = \sum_{i=1}^n (\mathbb{1}_{x + \mathbf{e}_i= \hat x} -  \mathbb{1}_{x = \hat x} )\,  dx_i \wedge dx_j
    \\&\qquad = \sum_{i=1}^{j-1} (\mathbb{1}_{x + \mathbf{e}_i = \hat x} -  \mathbb{1}_{x = \hat{x}} )\,  dx_i \wedge dx_j
    - \sum_{i=j+1}^{n} (\mathbb{1}_{x + \mathbf{e}_i = \hat x} -  \mathbb{1}_{x = \hat x} )\,    dx_j \wedge dx_i.
    \end{split}
\end{equation*}
Here we are writing $\mathbb{1}_{x = \hat x}$ for the Dirac delta function of $x$ with mass at $\hat x$.
From this equation, it follows that whenever \(  e = \hat{x} + d\mathbf{e}_j = d\hat x_j \) is an oriented edge and \( p \) is an oriented plaquette, we have
\begin{equation}\label{d1dxjp}
     \bigl(d(\mathbb{1}_{x = \hat x} \, dx_j )\bigr)_p =
     \begin{cases}1 &\text{if }  e \in \partial p, \cr 
     -1 &\text{if } - e \in \partial p,
     \cr 0 &\text{else.} \end{cases}
\end{equation}
 Note that this implies in particular that if \( 1 \leq j_1 < j_2 \leq n \), \( p =   dx_{j_1} \wedge dx_{j_2}  \) is a plaquette, and \( f \) is a \( 1 \)-form, then
 \begin{equation*}
     (df)_p = (df)_{j_1,j_2}(x) = \sum_{e \in \partial p}  f_e.
 \end{equation*}
Analogously, if \( k \in \{ 1,2, \ldots, n \} \) and \(  c  \) is a \( k \)-cell, we define \( \partial  c \) as the set of all \( (k-1) \)-cells \( \hat c =   d\hat x_{j_1} \wedge \cdots \wedge d\hat x_{j_{k-1}}\) such that
  \begin{equation*}
       \bigl(d(\mathbb{1}_{x = \hat x} \, d\hat x_{j_1} \wedge \dots \wedge d\hat x_{j_{k-1}}) \bigr)_{ c} = 1.
\end{equation*}
Using this notation, one can show that if \( f \) is a \( k \)-form and \( c_0 \) is a \( (k+1)\)-cell, then
\begin{equation*}
    (d f)_{c_0} = \sum_{c \in  \partial c_0} f_c.
\end{equation*}
If \( k \in \{ 1,2, \ldots, n \} \) and \( \hat c \) is a \( k \)-cell, the set \( \partial \hat c \) will be referred to as the \emph{boundary} of \( \hat c \). When \( k \in \{ 0,1,2,3, \ldots, n-1 \} \) and \( c \) is a \( k \)-cell, we also define the \emph{co-boundary} \( \hat \partial c \) of \( c \) as the set of all \( (k+1) \)-cells \( \hat c \) such that \( c \in \partial \hat c \).

Finally, when \( k \in \{ 1,2, \ldots, n \} \) and \( c \) is a \( k \)-cell, we will abuse notation and let
\begin{equation*}
    \hat \partial (\partial c) \coloneqq \bigcup_{c' \in \partial c} \hat \partial c'.
\end{equation*}
Similarly, when \( k \in \{ 0,1,2, \ldots, n-1 \} \) and \( c \) is a \( k \)-cell, we let
\begin{equation*}
     \partial (\hat \partial c) \coloneqq \bigcup_{c' \in \hat \partial c}  \partial c'.
\end{equation*}

The following lemma will be useful to us.

\begin{lemma}[The Bianchi lemma, see e.g. Lemma~2.5 in~\cite{flv2020}]\label{lemma: Bianchi}
    Let \( B \) be a box in \( \mathbb{Z}^n \) for some \( n \geq 3 \), and let \( \omega \in \Sigma_{P_B} \). Then, for any oriented  3-cell \( c \) in \( B \), we have  
    \begin{equation}\label{eq: Bianchi}
        \sum_{p \in \partial c} \omega_p = 0.
    \end{equation}
\end{lemma}

\subsubsection{Boundary cells}

Let \( k \in \{ 1,2,\ldots, n \} \). 
When \( A \) is a symmetric set of \( k \)-cells, we define the boundary of \( A \) by
\begin{equation*}
    \delta A \coloneqq \{ c \in A \colon \partial \hat \partial c \not \subseteq A \}.
\end{equation*}

If \( B \) is a box and  \( e \in \delta E_B \), we say that \( e \) is a \emph{boundary edge} of \( B \).
Analogously, a plaquette \( p \in \delta P_B \) is said to be a \emph{boundary plaquette} of \( B \).
More generally, for \( k \in \{ 0,1, \ldots, n-1 \} \),  a \( k \)-cell \( c \) in \( B \) is said to be a \emph{boundary cell} of \( B \), or equivalently to be in \emph{the boundary} of \( B \), if there is a \( (k+1) \)-cell \( \hat c \in \hat \partial c \) which contains a \( k \)-cell that is not in \( B \).

\subsubsection{Closed forms}
If \( k \in \{ 1,\ldots, n-1 \} \) and \( f \) is a \( k \)-form such that \( df = 0 \), then we say that \( f \) is \emph{closed}.

\subsubsection{The Poincar\'e lemma}

\begin{lemma}[The Poincar\'e lemma, Lemma 2.2 in~\cite{c2018b}]\label{lemma: poincare}
    Let \( k \in \{ 0,1, \ldots, n-1\} \) and let \( B \) be a box in \( \mathbb{Z}^n \). Then the exterior derivative \( d \) is a surjective map from the set of \(G\)-valued \( k \)-forms with support contained in \(B\) onto the set of G-valued closed \((k+1)\)-forms with support contained in \(B\). Moreover, if \(G\) is finite and \( m \) is the number of closed \(G\)-valued \( k \)-forms with support contained in \(B\), then this map is an \(m\)-to-\(1\) correspondence. Lastly, if \( k \in \{ 0,1,2, \ldots, n-1 \} \) and \(f\) is a closed \( (k+1) \)-form that vanishes on the boundary of \(B\), then there is a \(k\)-form \( h\) that also vanishes on the boundary of \(B\) and satisfies \(dh = f\).
\end{lemma}

Recall from the introduction that when \( B \) is a box in \( \mathbb{Z}^n \), we defined
\begin{equation*}
    \Sigma_{P_B} = \bigl\{ \omega \in \Sigma_{B,2} \colon \exists \sigma \in \Sigma_{E_B} \text{ such that } \omega = d\sigma \bigr\}.
\end{equation*}
From Lemma~\ref{lemma: poincare}, it follows that \( \omega \in \Sigma_{P_B} \) if and only if \( d\omega = 0 \).

\subsubsection{Restrictions of forms}\label{section: restrictions}

If \( \sigma \in \Sigma_{E_B} \), \( E \subseteq E_B \) is symmetric,  we define \( \sigma|_E \in \Sigma_{E_B} \) for \( e \in E_B \) by
\begin{equation*}
    (\sigma|_E)_e \coloneqq \begin{cases}
    \sigma_e &\text{if } e \in E  \cr 0 &\text{else.} 
    \end{cases} 
\end{equation*}
Similarly, if \( \nu \in \Sigma_{P_B} \), \( P \subseteq P_B \) is symmetric,  we define \( \omega|_P \in \Sigma_{B,2} \) for \( p \in P_B \) by
\begin{equation*}
    (\nu|_P)_p \coloneqq \begin{cases}
    \nu_p &\text{if } p \in P  \cr 0 &\text{else.}
    \end{cases}
\end{equation*}

\subsubsection{Non-trivial forms}

Let \( k \in \{ 1,2, \ldots, n \} \). A \( k \)-form \( f \) is said to be \emph{non-trivial} if it is not identically equal to zero.

\subsubsection{Irreducible forms}

Let \( B \) be a box in \( \mathbb{Z}^n \),  and let \( \omega \in \Sigma_{P_B}\).
If \( P \subseteq \support \omega\) is symmetric, and there is no symmetric set \( P_0 \subseteq P_B \) such that
\begin{enumerate}[label=(\roman*)]
    \item \( P \subseteq P_0 \subsetneq \support \omega \), and
    \item    \( \omega|_{P_0} \in \Sigma_{P_B}  \) (equivalently \( d(\omega|_{P_0}) =0 \)),
\end{enumerate}
then \( \omega \) is said to be \emph{\( P \)-irreducible}. If \( \omega \in \Sigma_{P_B}\) is \( P \)-irreducible for all non-empty, symmetric sets \( P \subseteq \support \omega \), then we say that \( \omega \) is \emph{irreducible}.
Finally, note that if \( \omega \in \Sigma_{P_B} \) is \( \emptyset \)-irreducible, then \( \omega \equiv 0 \).

\begin{lemma}\label{lemma: irreducible form relative set}
    Let \( B \) be a box in \( \mathbb{Z}^n \),  let \( \omega \in \Sigma_{P_B} \), and let \( P \subseteq \support \omega \) be non-empty and symmetric. Then there is a symmetric set \( P' \subseteq \support \omega\)  such that \( \omega|_{P'} \in \Sigma_{P_B} \) and  \( \omega|_{P'} \) is \( P \)-irreducible. 
\end{lemma}

\begin{proof}
    Consider the set \( \mathcal{S} \) of all symmetric sets \( P' \subseteq P_B \) which are such that 
    \begin{enumerate}[label=(\arabic*)] 
        \item \( P \subseteq P' \subseteq \support \omega \), and 
        \item \(  \omega|_{P'} \in \Sigma_{P_B}\). 
    \end{enumerate}
    Since  \( \support \omega \) is symmetric and \( \omega \in \Sigma_{P_B} \), we have   \(    \omega \in \mathcal{S}\), and hence \( \mathcal{S} \) is non-empty. Moreover, if we order the elements in \( \mathcal{S} \) using set inclusion, \( \mathcal{S} \) is a partially ordered set. Since \( P_B \) is finite, \( \mathcal{S} \) is finite, and hence there is a minimal element \( P' \in \mathcal{S} \). By definition, any such minimal element is \( P\)-irreducible, and hence the desired conclusion follows. 
\end{proof}

\subsubsection{The dual lattice}

The lattice \( \mathbb{Z}^n \) has a natural dual, called the \( \emph{dual lattice} \) and denoted by \( *\mathbb{Z}^n \). In this context, the lattice \( \mathbb{Z}^n \) is called the \emph{primal lattice.} 

The vertices of the dual lattice \( *\mathbb{Z}^n \) are placed at the centers of the \( n \)-cells of the primal lattice.

For \( k \in \{ 0,1,\ldots, n \} \), there is a bijection between the set of \( k \)-cells of \( \mathbb{Z}^n \) and the set of \( (n-k) \)-cells of \(*\mathbb{Z}^n \) defined as follows. For each \( x \in \mathbb{Z}^n \), let \( y \coloneqq *(dx_1 \wedge \cdots \wedge dx_n) \in *\mathbb{Z}^n\) be the point at the centre of the primal lattice \( n \)-cell \( dx_1 \wedge \cdots \wedge dx_n \).
Let \( dy_1 = y-d\mathbf{e}_1, \ldots, dy_n = y-d\mathbf{e}_n \) be the edges coming out of \( y \) in the \emph{negative} direction.
Next,  let \( k \in \{ 0,1, \ldots, n \} \) and assume that \( 1 \leq i_1 < \cdots < i_k \leq n \) are given. If $x \in \mathbb{Z}^n$, then \( c = dx_{i_1} \wedge \cdots \wedge dx_{i_k} \) is a \( k \)-cell in \( \mathbb{Z}^n \).  Let \( {j_1}, \ldots, j_{n-k} \) be any enumeration of \( \{ 1,2, \ldots, n \} \smallsetminus \{ i_1, \ldots, i_k \} \), and let \(  \sgn (i_1,\ldots, i_k, j_{1}, \ldots, j_{n-k} )\) denote the sign of the permutation that maps $(1,2,\ldots, n)$ to \( (i_1,\ldots, i_k, j_{1}, \ldots, j_{n-k} ) \). Define 
\begin{equation*}
    *(dx_{i_1} \wedge \cdots \wedge dx_{i_k}) =  \sgn (i_1,\ldots, i_k, j_{1}, \ldots, j_{n-k} )\, dy_{j_1} \wedge \cdots \wedge dy_{j_{n-k}}
\end{equation*}
and, analogously, define
\begin{equation*}
    \begin{split}
         *&(dy_{j_1} \wedge \cdots \wedge dy_{j_{n-k}}) =
        \sgn (j_{1}, \ldots, j_{n-k},i_1,\ldots, i_k)\, dx_{i_1} \wedge \cdots \wedge dx_{i_k}
        \\&\qquad =
        (-1)^{k(n-k)}  \sgn (i_1,\ldots, i_k, j_{1}, \ldots, j_{n-k} )\, dx_{i_1} \wedge \cdots \wedge dx_{i_k}.
    \end{split}
\end{equation*}

\subsubsection{Minimal non-trivial configurations}

The purpose of the next two lemmas is to describe the non-trivial plaquette configurations in \( \Sigma_{P_B} \) with smallest support.

\begin{lemma}\label{lemma: small vortices}
    Let \( B \) be a box in \( \mathbb{Z}^4 \), and let \( \omega \in \Sigma_{P_B} \). If \( \omega \neq 0 \) and the support of \( \omega \) does not contain any boundary plaquettes of \( P_B \), then either \( |\support \omega| = 12 \), or \( |\support \omega| \geq 22 \). 
\end{lemma}

\begin{proof}
    Let \( P \coloneqq \support \omega \). 
    By Lemma~\ref{lemma: Bianchi}, if \( c \) is an oriented 3-cell in the primary lattice, then
    \begin{equation*}
        |\partial c \cap P| \in \{ 0,2,3,4 \}.
    \end{equation*}
    Consequently, if \( e = *c \) in an oriented edge in the dual lattice, then
    \begin{equation}\label{eq: count}
        |\hat \partial e \cap *P| \in \{ 0,2,3,4 \}.
    \end{equation}
     
    Let \(  \overline{* P} \) be the set of unoriented plaquettes obtained from \( *P \) by identifying \( p \) and \( -p \) for each \( p \in *P \). It then follows from~\eqref{eq: count} that each unoriented edge \( \bar e \) in the dual lattice which is in the boundary of a plaquette in \( \overline{* P} \) must be in the boundary of at least two plaquettes in \( \overline{* P} \). 
    In other words, the set \( \overline{* P} \) is a closed surface in the dual lattice. One easily verifies that the closed (non-empty) surfaces in the dual lattice which contains the fewest number of plaquettes are 3-dimensional cubes (see Figure~\ref{table: minimal configurations}), and hence we must have \( |\overline{* P} | \geq 6 \). If \( |\overline{* P}  | > 6 \), then by the same argument we must have \( |* \bar P| \geq 11 \) (see Figure~\ref{table: minimal configurations}).
    Since \( |\support \omega| = |P| =  2|\overline{* P} | \), the desired conclusion follows.
\end{proof}

\begin{figure}[ht]
    \centering
    \begin{tabular}{p{2.5cm} p{4cm} p{4cm} m{1cm}}
        \( \overline{\support \sigma} \) & \(  \bar P = \overline{\support d\sigma} \) & \( \overline{* P}\) & \( |\bar P| \)
        \\[0.5ex] \toprule \\[-1ex]
        \begin{tikzpicture}
            \draw[white] (-0.5,1.3) -- (0.5,-0.3);
            \draw[thick, Red] (0,0) -- (0,1);
        \end{tikzpicture} & 
        \begin{tikzpicture}
            \filldraw[fill=SkyBlue, fill opacity=0.12] (0,0) -- (1,0) -- (1,1) -- (0,1) -- (0,0);
            \filldraw[fill=SkyBlue, fill opacity=0.12] (0,0) -- (0,1) -- (-1,1) -- (-1,0) -- (0,0);
            \filldraw[fill=SkyBlue, fill opacity=0.12] (0,0) -- (0.55,-0.3) -- (0.55,0.7) -- (0,1) -- (0,0);
            \filldraw[fill=SkyBlue, fill opacity=0.12] (0,0) -- (0,1) -- (-0.55,1.3) -- (-0.55,0.3) -- (0,0);
            \filldraw[fill=SkyBlue, fill opacity=0.12] (0,0) -- (0.7,0.25) -- (0.7,1.25) -- (0,1) -- (0,0);
            \filldraw[fill=SkyBlue, fill opacity=0.12] (0,0) -- (0,1) -- (-0.7,0.75) -- (-0.7,-0.25) -- (0,0);
            \draw[thick, Red] (0,0) -- (0,1);
        \end{tikzpicture} &  
        \begin{tikzpicture}
            \draw[white] (0,1.3) -- (0,-0.2);
            
            \fill[fill= SkyBlue, opacity=0.12] (0,0) -- (1,0) -- (1,1) -- (0,1) -- (0,0);
            \fill[fill= SkyBlue, opacity=0.12] (0.55,0.3) -- (1.55,0.3) -- (1.55,1.3) -- (0.55,1.3) -- (0.55,0.3); 
            \fill[fill= SkyBlue, opacity=0.12](0,0) -- (0.55,0.3) -- (1.55,0.3) -- (1,0) -- (0,0);  
            \fill[fill= SkyBlue, opacity=0.12] (0,1) -- (0.55,1.3) -- (1.55,1.3) -- (1,1) -- (0,1); 
            \fill[fill= SkyBlue, opacity=0.12] (0,0) -- (0.55,0.3) -- (0.55,1.3) -- (0,1) -- (0,0); 
            \fill[fill= SkyBlue, opacity=0.12] (1,0) -- (1.55,0.3) -- (1.55,1.3) -- (1,1) -- (1,0);

            \draw (0,0) -- (1,0) -- (1,1) -- (0,1) -- (0,0);
            \draw (0.55,0.3) -- (1.55,0.3) -- (1.55,1.3) -- (0.55,1.3) -- (0.55,0.3);
            \draw (0,0) -- (0.55,0.3);
            \draw (1,0) -- (1.55,0.3);
            \draw (1,1) -- (1.55,1.3);
            \draw (0,1) -- (0.55,1.3);
            
        \end{tikzpicture} &
        6\vspace{5.5ex}
        \\  
        \begin{tikzpicture}
            \draw[white] (0,1.3) -- (0,-0.3);
            
            \draw[thick, Red] (0,0) -- (0,1);
            \draw[thick, Red] (1,0) -- (1,1);
        \end{tikzpicture} & 
        \begin{tikzpicture}  
            
            \filldraw[fill=SkyBlue, fill opacity=0.12] (0,0) -- (0,1) -- (-1,1) -- (-1,0) -- (0,0);
            \filldraw[fill=SkyBlue, fill opacity=0.12] (0,0) -- (0.55,-0.3) -- (0.55,0.7) -- (0,1) -- (0,0);
            \filldraw[fill=SkyBlue, fill opacity=0.12] (0,0) -- (0,1) -- (-0.55,1.3) -- (-0.55,0.3) -- (0,0);
            \filldraw[fill=SkyBlue, fill opacity=0.12] (0,0) -- (0.7,0.25) -- (0.7,1.25) -- (0,1) -- (0,0);
            \filldraw[fill=SkyBlue, fill opacity=0.12] (0,0) -- (0,1) -- (-0.7,0.75) -- (-0.7,-0.25) -- (0,0);
            
            \filldraw[fill=SkyBlue, fill opacity=0.12] (1,0) -- (2,0) -- (2,1) -- (1,1) -- (1,0);
            \filldraw[fill=SkyBlue, fill opacity=0.12] (1,0) -- (1.55,-0.3) -- (1.55,0.7) -- (1,1) -- (1,0);
            \filldraw[fill=SkyBlue, fill opacity=0.12] (1,0) -- (1,1) -- (0.45,1.3) -- (0.45,0.3) -- (1,0);
            \filldraw[fill=SkyBlue, fill opacity=0.12] (1,0) -- (1.7,0.25) -- (1.7,1.25) -- (1,1) -- (1,0);
            \filldraw[fill=SkyBlue, fill opacity=0.12] (1,0) -- (1,1) -- (0.3,0.75) -- (0.3,-0.25) -- (1,0);
            
            \draw[thick, Red] (0,0) -- (0,1);
            \draw[thick, Red] (1,0) -- (1,1);
            
        \end{tikzpicture} &  
        \begin{tikzpicture}
            \draw[white] (0,1.3) -- (0,-0.2);

            \fill[fill= SkyBlue, opacity=0.12] (0,0) -- (1,0) -- (1,1) -- (0,1) -- (0,0);
            \fill[fill= SkyBlue, opacity=0.12] (0.55,0.3) -- (1.55,0.3) -- (1.55,1.3) -- (0.55,1.3) -- (0.55,0.3); 
            \fill[fill= SkyBlue, opacity=0.12](0,0) -- (0.55,0.3) -- (1.55,0.3) -- (1,0) -- (0,0);  
            \fill[fill= SkyBlue, opacity=0.12] (0,1) -- (0.55,1.3) -- (1.55,1.3) -- (1,1) -- (0,1); 
            \fill[fill= SkyBlue, opacity=0.12] (0,0) -- (0.55,0.3) -- (0.55,1.3) -- (0,1) -- (0,0); 
            
            \fill[fill= SkyBlue, opacity=0.12] (1,0) -- (2,0) -- (2,1) -- (1,1) -- (1,0);
            \fill[fill= SkyBlue, opacity=0.12] (1.55,0.3) -- (2.55,0.3) -- (2.55,1.3) -- (1.55,1.3) -- (1.55,0.3); 
            \fill[fill= SkyBlue, opacity=0.12] (1,0) -- (1.55,0.3) -- (2.55,0.3) -- (2,0) -- (1,0);  
            \fill[fill= SkyBlue, opacity=0.12] (1,1) -- (1.55,1.3) -- (2.55,1.3) -- (2,1) -- (1,1); 
            \fill[fill= SkyBlue, opacity=0.12] (2,0) -- (2.55,0.3) -- (2.55,1.3) -- (2,1) -- (2,0);

            \draw (0,0) -- (2,0) -- (2,1) -- (0,1) -- (0,0);
            \draw (0,0) -- (0.55,0.3) -- (2.55,0.3) -- (2,0); 
            \draw (0,0) -- (0,1) -- (0.55,1.3) -- (2.55,1.3) -- (2,1); 
            \draw (0.55,0.3) -- (0.55,1.3);
            \draw (2.55,0.3) -- (2.55,1.3);
            
            \draw (1,0) -- (1,1) -- (1.55,1.3) -- (1.55,0.3) -- (1,0);
        \end{tikzpicture} &
        11\vspace{6.5ex}
        \\  
        \begin{tikzpicture}
            \draw[white] (0.,1.8) -- (0.,-0.9);
            
            \draw[thick, Red] (1,0) -- (0,0) -- (0,1);
        \end{tikzpicture} & 
        \begin{tikzpicture}
         
            \draw[white] (0,1.3) -- (0,-1.3);  
            
            \filldraw[fill=SkyBlue, fill opacity=0.12] (0,0) -- (0,1) -- (-1,1) -- (-1,0) -- (0,0);
            \filldraw[fill=SkyBlue, fill opacity=0.12] (0,0) -- (0.55,-0.3) -- (0.55,0.7) -- (0,1) -- (0,0);
            \filldraw[fill=SkyBlue, fill opacity=0.12] (0,0) -- (0,1) -- (-0.55,1.3) -- (-0.55,0.3) -- (0,0);
            \filldraw[fill=SkyBlue, fill opacity=0.12] (0,0) -- (0.7,0.25) -- (0.7,1.25) -- (0,1) -- (0,0);
            \filldraw[fill=SkyBlue, fill opacity=0.12] (0,0) -- (0,1) -- (-0.7,0.75) -- (-0.7,-0.25) -- (0,0); 
            
            \filldraw[fill=SkyBlue, fill opacity=0.12] (0,0) -- (1,0) -- (1,-1) -- (0,-1) -- (0,0);
            \filldraw[fill=SkyBlue, fill opacity=0.12] (0,0) -- (0.55,-0.3) -- (1.55,-0.3) -- (1,0) -- (0,0);
            \filldraw[fill=SkyBlue, fill opacity=0.12] (0,0) -- (-0.55,0.3) -- (0.45,0.3) -- (1,0) -- (0,0);
            \filldraw[fill=SkyBlue, fill opacity=0.12] (0,0) -- (-0.7,-0.25) -- (0.3,-0.25) -- (1,0) -- (0,0);
            \filldraw[fill=SkyBlue, fill opacity=0.12] (0,0) -- (0.7,0.25) -- (1.7,0.25) -- (1,0) -- (0,0);
            
            \draw[thick, Red] (1,0) -- (0,0) -- (0,1);
            
        \end{tikzpicture} &  
        \begin{tikzpicture}
            \draw[white] (0,1.) -- (0,-0.9);
            
            \fill[fill= SkyBlue, opacity=0.12] (0,0) -- (1,0) -- (1,1) -- (0,1) -- (0,0);
            \fill[fill= SkyBlue, opacity=0.12] (0.55,0.3) -- (1.55,0.3) -- (1.55,1.3) -- (0.55,1.3) -- (0.55,0.3); 
            \fill[fill= SkyBlue, opacity=0.12](0,0) -- (0.55,0.3) -- (1.55,0.3) -- (1,0) -- (0,0);  
            \fill[fill= SkyBlue, opacity=0.12] (0,1) -- (0.55,1.3) -- (1.55,1.3) -- (1,1) -- (0,1); 
            \fill[fill= SkyBlue, opacity=0.12] (0,0) -- (0.55,0.3) -- (0.55,1.3) -- (0,1) -- (0,0);

            \fill[fill= SkyBlue, opacity=0.12] (1,0) -- (1.7,-0.4) -- (1.7,0.6) -- (1,1) -- (1,0); 
            \fill[fill= SkyBlue, opacity=0.12] (1,0) -- (1.7,-0.4) -- (2.25,-0.1) -- (1.55,0.3);
            \fill[fill= SkyBlue, opacity=0.12] (2.25,-0.1) -- (1.55,0.3) -- (1.55,1.3) -- (2.25,0.9);
            \fill[fill= SkyBlue, opacity=0.12] (1,1) -- (1.7,0.6) -- (2.25,0.9) -- (1.55,1.3);
            \fill[fill= SkyBlue, opacity=0.12] (2.25,-0.1) -- (2.25,0.9) -- (1.7,0.6) -- (1.7,-0.4);
            
            \draw (0,0) -- (1,0) -- (1,1) -- (0,1) -- (0,0) -- (0.55,0.3) -- (1.55,0.3) -- (1,0); 
            \draw (0,1) -- (0.55,1.3) -- (1.55,1.3) -- (1,1); 
            \draw (0.55,0.3) -- (0.55,1.3);
            \draw (1.55,0.3) -- (1.55,1.3);
            
            \draw (1,0) -- (1.7,-0.4) -- (1.7,0.6) -- (1,1);
            \draw (1.7,-0.4) -- (2.25,-0.1) -- (1.55,0.3);
            \draw (2.25,-0.1) -- (2.25,0.9)-- (1.55,1.3);
            \draw (2.25,0.9) -- (1.7,0.6);
        \end{tikzpicture} &
        11  \vspace{14ex}
\end{tabular} 
\vspace{-8ex}
\caption{The above table shows projections of the supports of the  non-trivial and irreducible plaquette configurations in \( \mathbb{Z}^4 \) which has the smallest support (up to translations and rotations), using the notation of the proof of Lemma~\ref{lemma: small vortices}.}\label{table: minimal configurations}
\end{figure}

\begin{lemma}[Lemma~4.6 in~\cite{flv2020}]\label{lemma: minimal vortices}
    Let \( B \) be a box in \( \mathbb{Z}^4 \), and let \( \omega \in \Sigma_{P_B} \).  If the support of \( \omega \) does not contain any boundary plaquettes of \( P_B \) and \( | \support \omega| = 12\), then there is an edge \( dx_j \in E_B \) and \( g \in G\smallsetminus \{ 0 \} \) such that 
        \begin{equation*}
            \omega = d\bigl(g \, dx_j \bigr). 
        \end{equation*} 
\end{lemma}

\subsubsection{Frustrated plaquettes }

When \( B \) is a box in \( \mathbb{Z}^4 \) and \( \omega \in \Sigma_{P_B} \), we say that a plaquette \( p \in P_B \) is \emph{frustrated} (in \( \omega \)) if \( \omega_p \neq 0 \).

\subsection{\texorpdfstring{\( \mu_{B,\beta} \)}{The measure} as a measure on plaquette configurations}

Let \( B \) be a box in \( \mathbb{Z}^4 \), and let \( \beta \geq 0\). 
In~\eqref{eq: Wilson action}, we introduced \( \mu_{B,\beta} \) as a measure on \( \Sigma_{E_B} \). Using Lemma~\ref{lemma: poincare}, this induces a measure on \( \Sigma_{P_B} \) as follows. For \( \omega \in \Sigma_{P_B} \), by definition, we have 
\begin{equation*}
    \mu_{B,\beta}(\{\sigma \in \Sigma_{E_B} \colon d\sigma = \omega\})
    =
    \frac{\sum_{\sigma \in \Sigma_{E_B}\colon d\sigma = \omega} \prod_{p \in P_B} \phi_\beta\bigl((d\sigma)_p \bigr)}{\sum_{\sigma \in \Sigma_{E_B}} \prod_{p \in P_B} \phi_\beta \bigl((d\sigma)_p \bigr) }.
\end{equation*}
If \( \sigma \in \Sigma_{E_B} \), then \( d\sigma \in \Sigma_{P_B} \). Consequently, the previous equation is equal to
\begin{equation*}
    \frac{\sum_{\sigma \in \Sigma_{E_B}\colon d\sigma = \omega} \prod_{p \in P_B} \phi_\beta(\omega_p) }{\sum_{\omega' \in \Sigma_{P_B}}\sum_{\sigma \in \Sigma_{E_B} \colon d\sigma=\omega'} \prod_{p \in P_B} \phi_\beta(\omega'_p)} 
\end{equation*}
Changing the order of summation, we get
\begin{equation*} 
    \frac{ \bigl( \prod_{p \in P_B} \phi_\beta(\omega_p)\bigr)  \bigl| \{\sigma \in \Sigma_{E_B}\colon d\sigma = \omega\}\bigr| }{ \bigl(\sum_{\omega' \in \Sigma_{P_B}}\prod_{p \in P_B} \phi_\beta(\omega_p') \bigr)
    \bigl|\{\sigma \in \Sigma_{E_B} \colon d\sigma=\omega'\} \bigr| }
\end{equation*}
By Lemma~\ref{lemma: poincare}, the term \( |\{\sigma \in \Sigma_{E_B} \colon d\sigma=\omega'\}| \) is equal for all \( \omega' \in \Sigma_{P_B} \), and hence in particular, for all \( \omega' \in \Sigma_{P_B} \) we have \( |\{\sigma \in \Sigma_{E_B} \colon d\sigma=\omega'\}| =|\{\sigma \in \Sigma_{E_B} \colon d\sigma=\omega\}| \). Combining the above equations, we thus obtain
\begin{equation}\label{eq: mu as a measure on plaq conf}
    \mu_{B,\beta}(\{\sigma \in \Sigma_{E_B} \colon d\sigma = \omega\})
    =
    \frac{\prod_{p \in P_B} \phi_\beta(\omega_p) }{\sum_{\omega' \in \Sigma_{P_B}} \prod_{p \in P_B} \phi_\beta(\omega'_p)  }.
\end{equation}
Consequently, \( \mu_{B,\beta} \) induces a measure on plaquette configurations.
In order to simplify notation, we will abuse notation and use \( \mu_{B,\beta} \) and \( \mathbb{E}_{B,\beta} \) for both the measure on \( \Sigma_{E_B} \) and for the induced measure on \( \Sigma_{P_B} \).

\subsection{The activity of plaquette configurations}

When \( B \) is a box in \( \mathbb{Z}^4 \), \( \omega \in P_B \) and \( \beta \geq 0 \), then, recalling the definition of \( \phi_\beta \) from~\eqref{eq: phi def}, we abuse notation and write
\begin{equation*}
    \phi_\beta(\omega) \coloneqq \prod_{p \in \support \omega} \phi_\beta(\omega_p).
\end{equation*} 
The quantity \( \phi_\beta(\omega) \) is called the \emph{activity} of \( \omega \) (see e.g.~\cite{b1984}). Using this notation, since \( \phi_\beta(0)=1 \), for any \( \omega \in \Sigma_{P_B} \), we also have
\begin{equation*}
    \prod_{p\in P_B} \phi_{\beta}(\omega_p) 
    = 
    \prod_{p\in \support \omega} \phi_{\beta}(\omega_p) 
    =
    \phi_\beta(\omega) 
\end{equation*}
and hence, using~\eqref{eq: mu as a measure on plaq conf}, we can write
\begin{equation}\label{eq: mu using phi}
    \mu_{B,\beta} \bigl(\{\omega\} \bigr)  
    =
    \frac{\phi_\beta(\omega)}{\sum_{\omega'\in \Sigma_{P_B}} \phi_\beta(\omega')}.
\end{equation}

The following lemma is referred to as the factorization property of \( \phi_\beta \) in e.g.~\cite{cao2020}.
\begin{lemma}\label{lemma: factorization}
    Let \( B \) be a box in \( \mathbb{Z}^4 \), and let \( \beta \geq 0 \). Further, let \( \omega,\omega' \in \Sigma_{P_B} \) be such that \( \support \omega \cap \support \omega' = \emptyset \). Then
    \begin{equation*}
        \phi_\beta(\omega + \omega') = \phi_\beta(\omega) \phi_\beta(\omega').
    \end{equation*}
\end{lemma}

\begin{proof}
    By definition, we have
    \begin{equation*}
        \phi_\beta(\omega + \omega') 
        = \!\!\!\prod_{p \in \support (\omega + \omega')} \!\!\!\!\! \phi_\beta \bigl( (\omega + \omega')_p \bigr)
        = \!\!\!\prod_{p \in \support (\omega + \omega')}\!\!\!\!\! \phi_\beta \bigl( \omega _p+ \omega'_p \bigr).
    \end{equation*}
    Since \( \omega \) and \( \omega' \) have disjoint supports, we have
    \begin{equation*}
        \prod_{p \in \support (\omega + \omega')}\!\!\!\!\! \phi_\beta \bigl( \omega _p+ \omega'_p \bigr)
        = 
        \prod_{p \in \support \omega} \phi_\beta \bigl( \omega _p+ 0 \bigr)
        \prod_{p \in \support \omega'} \phi_\beta \bigl( 0+ \omega'_p \bigr)
        = 
        \prod_{p \in \support \omega} \phi_\beta \bigl( \omega _p\bigr)
        \prod_{p \in \support \omega'} \phi_\beta \bigl(  \omega'_p \bigr) .
    \end{equation*}
    Since by definition, we have
    \begin{equation*}
        \prod_{p \in \support \omega} \phi_\beta \bigl( \omega _p\bigr)
        \prod_{p \in \support \omega'} \phi_\beta \bigl(  \omega'_p \bigr) = \phi_\beta(\omega) \phi_\beta(\omega' ) ,
    \end{equation*}
    the desired conclusion immediately follows.
\end{proof}

Combining~\eqref{eq: mu using phi} and Lemma~\ref{lemma: factorization}, we obtain the following lemma.
This lemma can be extracted from the proof of Lemma~4.7 in~\cite{flv2020}, but we state and prove it here as an independent lemma for easier reference.

\begin{lemma}  \label{lemma: ratio}     
    Let \( B \) be a box in \( \mathbb{Z}^4 \), let \( \beta \geq 0 \), and let \( \nu \in \Sigma_{P_B} \). Then 
	\begin{equation}\label{eq: ratio in lemma} 
	    \frac{\mu_{B,\beta} \bigl( \{ \omega \in \Sigma_{P_B} \colon \omega|_{\support \nu} = \nu \} \bigr)}{\mu_{B,\beta} \bigl( \{ \omega \in \Sigma_{P_B} \colon \omega|_{\support \nu} = 0 \} \bigr)} =  \phi_\beta  (\nu).
	\end{equation} 
\end{lemma}

\begin{proof} 
    Let \(  P = \support \nu \). 
    Further, let 
    \begin{equation*}
        \mathcal{E}_{ P,\nu}  \coloneqq \bigl\{ \omega \in \Sigma_{P_B} \colon \omega|_{ P} = \nu \bigr\},
    \end{equation*}
    and, similarly, let 
    \begin{equation*}
         \mathcal{E}_{ P,0} \coloneqq \bigl\{ \omega \in \Sigma_{P_B} \colon \omega|_{ P} = 0 \bigr\}.
    \end{equation*}
    By~\eqref{eq: mu using phi},  we then have
    \begin{equation}\label{eq: eq 2 in proof i}
        \frac{\mu_{B,\beta} \bigl( \{ \omega \in \Sigma_{P_B} \colon \omega|_{\support \nu} = \nu \} \bigr)}{\mu_{B,\beta} \bigl( \{ \omega \in \Sigma_{P_B} \colon \omega|_{\support \nu} = 0 \} \bigr)}
        =
        \frac{\sum_{\omega \in \mathcal{E}_{ P,\nu}}  \phi_\beta  (\omega)}{\sum_{\omega \in \mathcal{E}_{ P,0}}  \phi_\beta  (\omega)}.
    \end{equation}
    Since \( \nu \in \Sigma_{P_B} \) by assumption, we have \( d\nu = 0 \). Consequently, for any \( \omega \in \Sigma_{P_B} \), we have
    \begin{equation*}
        d(\omega-\nu) = d\omega - d\nu = 0-0 = 0,
    \end{equation*}
    and hence \( (\omega-\nu) \in \Sigma_{P_B} \). This implies in particular that the mapping \( \omega \mapsto \omega-\nu \) is a bijection from \( \mathcal{E}_{ P,\nu} \) to \( \mathcal{E}_{ P,0} \), and hence the right-hand side of~\eqref{eq: eq 2 in proof i} is equal to
    \begin{equation}  \label{eq: second eq}
        \frac{\sum_{\omega \in \mathcal{E}_{ P,\nu}}  \phi_\beta  (\omega)}{\sum_{\omega \in \mathcal{E}_{ P,\nu}}  \phi_\beta  (\omega-\nu)}
        = 
        \frac{\sum_{\omega \in \mathcal{E}_{ P,\nu}}  \phi_\beta  ((\omega-\nu) + \nu)}{\sum_{\omega \in \mathcal{E}_{ P,\nu}}  \phi_\beta  (\omega-\nu)}.
    \end{equation}
   Next, note that if \( \omega \in \mathcal{E}_{P,\nu}   \), then  \( (\omega- \nu) \) and \( \nu \) have disjoint supports. Consequently, for such \( \omega \) we can apply Lemma~\ref{lemma: factorization} to obtain
    \begin{equation*}
        \phi_\beta((\omega-\nu)+\nu) = \phi_\beta(\omega-\nu) \phi_\beta(\nu).
    \end{equation*}
    Plugging this into the left hand side of~\eqref{eq: second eq}, we obtain
    \begin{equation*}  
        \frac{\sum_{\omega \in \mathcal{E}_{ P,\nu}}  \phi_\beta  ((\omega-\nu) + \nu)}{\sum_{\omega \in \mathcal{E}_{ P,\nu}}  \phi_\beta  (\omega-\nu)}
        =
        \frac{\sum_{\omega \in \mathcal{E}_{ P,\nu}}  \phi_\beta  (\omega-\nu) \phi_\beta( \nu)}{\sum_{\omega \in \mathcal{E}_{ P,\nu}}  \phi_\beta  (\omega-\nu)} = \phi_\beta(\nu).
    \end{equation*} 
    Combining the previous equations, we obtain~\eqref{eq: ratio in lemma} as desired.
\end{proof}

\subsection{Notation and standing assumptions}

Throughout the remainder of this paper, we assume that a finite Abelian structure group \( G \), and a faithful, irreducible and unitary representation of \( \rho \) has been fixed.

\section{The probability of sets of plaquettes being frustrated}\label{sec: frustrated sets}
 
In this section we will state and prove three propositions which will be useful in later sections. We mention that although these are technical, they might be useful outside the scope of this paper. 
The first of these results is the following proposition, which  extends Proposition~4.9 in~\cite{flv2020}.

\begin{proposition} \label{proposition: vortex probability}
    Let \( B \) be a box in \( \mathbb{Z}^4 \),  let  \( \beta \geq 0 \) be such that \( 5\alpha(\beta)< 1 \), and let \( P \subseteq  P_B\) be non-empty and symmetric. For \( M \geq |P^+| \), let \( \Pi_{P, M}^\geq \coloneqq \Pi_{B,P, M}^\geq \) be the set of all    \( \omega \in \Sigma_{P_B} \)  such that there is \( \nu \in \Sigma_{P_B} \) with
    \begin{enumerate}[label=(\arabic*), series=vortexfliplemma]
        \item \( P \subseteq \support \nu \), \label{enum: vortex probability ii}
        \item \( \nu \) is \( P \)-irreducible,  \label{enum: vortex probability iii}   
        \item \( |\support \nu | \geq 2M  \), and \label{enum: vortex probability iv}
        \item \( \omega|_{\support \nu} = \nu \). \label{enum: vortex probability i}
    \end{enumerate}
    Then
    \begin{equation*}
        \mu_{\beta,N}\bigl( \Pi_{P,M}^\geq) \leq   \frac{5^{M-|P^+|} \alpha(\beta)^{M}}{1 - 5\alpha(\beta) } .
    \end{equation*} 
\end{proposition}

Before we give a proof of Proposition~\ref{proposition: vortex probability}, we state and prove a few lemmas which will be used in the proof. 
The first of these lemmas is Lemma~\ref{lemma: agreement probability} below, which essentially is identical to Lemma~4.7 in~\cite{flv2020}.
\begin{lemma} \label{lemma: agreement probability}     
    Let \( B \) be a box in \( \mathbb{Z}^4 \), let \( \beta \geq 0 \), and let \( \nu \in \Sigma_{P_B} \). Then 
	\begin{equation}\label{eq: agreement probability}  
	    \mu_{B,\beta} \bigl( \{ \omega \in \Sigma_{P_B} \colon \omega|_{\support \nu} = \nu \} \bigr) \leq  \phi_\beta  (\nu) .
	\end{equation} 
\end{lemma}

\begin{proof}  
    Since 
	\begin{align*}
	    &\mu_{B,\beta} \bigl( \{ \omega \in \Sigma_{P_B} \colon \omega|_{\support \nu} = \nu \} \bigr) 
	    \\&\qquad=
	    \frac{
	    \mu_{B,\beta} \bigl( \{ \omega \in \Sigma_{P_B} \colon \omega|_{\support \nu} = \nu \} \bigr) }{
	    \mu_{B,\beta} \bigl( \{ \omega \in \Sigma_{P_B} \colon \omega|_{\support \nu} = 0 \} \bigr) } \cdot 
	    \mu_{B,\beta} \bigl( \{ \omega \in \Sigma_{P_B} \colon \omega|_{\support \nu} = 0 \} \bigr),
	\end{align*} 
	and \( \mu_{B,\beta} \) is a probability measure, we have
	\begin{align*}
	    &\mu_{B,\beta} \bigl( \{ \omega \in \Sigma_{P_B} \colon \omega|_{\support \nu} = \nu \} \bigr) 
	    \leq
	    \frac{
	    \mu_{B,\beta} \bigl( \{ \omega \in \Sigma_{P_B} \colon \omega|_{\support \nu} = \nu \} \bigr) }{
	    \mu_{B,\beta} \bigl( \{ \omega \in \Sigma_{P_B} \colon \omega|_{\support \nu} = 0 \} \bigr) }.
	\end{align*} 
    Using Lemma~\ref{lemma: ratio}, we obtain~\eqref{lemma: agreement probability}  as desired.
\end{proof}

The next result we will need in the proof of Proposition~\ref{proposition: vortex probability} is the following lemma, which gives an upper bound on the sum of the activity of the  plaquette configurations \( \nu \) which satisfies~\ref{enum: vortex probability ii},~\ref{enum: vortex probability iii}, and ~\ref{enum: vortex probability iv} of Proposition~\ref{proposition: vortex probability}.

\begin{lemma}\label{lemma: counting vortex configurations}
Let \( B \) be a box in \( \mathbb{Z}^4 \), and let \( \beta \geq 0 \). Further, let \(  P\subseteq  P_B \) be  non-empty and symmetric,  let \( m \geq |P^+| \), and let \( \Pi_{P,m} \) be the set of all plaquette configurations \( \nu \in \Sigma_{P_B} \) such that 
\begin{enumerate}[label=(\arabic*)]
    \item \( P \subseteq \support \nu  \), \label{enum: no of configs i}
    \item \( \nu \) is \( P \)-irreducible \label{enum: no of configs iii}
    \item \( |\support \nu| = 2m \). \label{enum: no of configs ii}
\end{enumerate}
Then
\begin{equation*}
    \sum_{\nu \in \Pi_{P,m}}  \phi_\beta  (\nu)\leq 5^{m-|P^+|} \alpha(\beta)^m.
\end{equation*}
\end{lemma}

\begin{remark}
Lemma~\ref{lemma: counting vortex configurations} is very similar to Lemma~4.9 in~\cite{flv2020}, and is also similar to Lemma~3.11~in~\cite{s1982} and Lemma~1.5~in~\cite{b1984}. We remark however that our proof strategy yields a strictly better upper bound than the corresponding proofs in~\cite{b1984} and~\cite{s1982}. This would yield a strictly larger lower bound on \( \beta \) in the main results of this paper, and hence we give an alternative proof here.
\end{remark}

\begin{proof}[Proof of Lemma~\ref{lemma: counting vortex configurations}]
    Let \( \ell \coloneqq |P_+| \).
    We will prove that Lemma~\ref{lemma: counting vortex configurations} holds by giving a injective map from the set \( \Pi_{P,m} \) to a set of sequences \( \nu^{(\ell)} ,\nu^{(\ell+1)}, \ldots, \nu^{(m)} \) of \( G \)-valued 2-forms on \( P_B \), and then use this map to obtain the desired upper bound.
    To this end, assume that a total ordering of the plaquettes in \( P_B \) and a total ordering of the 3-cells in \( B \) are given.
    If \( \Pi_{P,m} \) is empty, then the desired conclusion trivially holds, and hence we can assume that this is not the case.

    Fix some \( \nu \in \Pi_{P,m} \).
    Let \( \{ p_1,p_2,\ldots , p_\ell \} \coloneqq P^+ \), and define 
        \begin{equation*}
            \nu^{(\ell)} \coloneqq \nu|_P.
        \end{equation*}
        
    Now assume that for some \( k\in \{ \ell,\ell+1, \ldots, m\} \), we are given 2-forms \( \nu^{(\ell)}, \nu^{(\ell+1)} , \ldots, \nu^{(k)} \)  such that 
\begin{enumerate}[label=(\roman*)]
    \item for each \( j \in \{ \ell+1, \ell+2, \ldots, k \} \), we have \( \support \nu^{(j)} = \support \nu^{(j-1)} \sqcup \{ p_j,-p_j \} \) for some \( p_j \in P_B \), and \label{enum: no of configs proof i}
    \item for each \( j \in \{ \ell, \ell+1, \ldots, k \} \) we have \( \nu|_{\support \nu^{(j)}} = \nu^{(j)} \). \label{enum: no of configs proof ii}
\end{enumerate}

Consider first the case that \( d\nu^{(k)} = 0 \). Since, by~\ref{enum: no of configs proof ii}, we have \( \nu|_{\support \nu^{(k)}} = \nu^{(k)} \), it follows from~\ref{enum: no of configs iii} that we must have \( \nu^{(k)} = \nu \).
On the other hand, by~\ref{enum: no of configs proof i}, we have  \( |\support \nu^{(k)}| = 2k \), and hence using~\ref{enum: no of configs ii} we obtain  \( k = m \).
Consequently, if \( k <m \), then \( d\nu^{(k)} \not \equiv 0 \). Equivalently, in this case there is at least one oriented 3-cell \( c  \) in \( B \) for which \( (d \nu^{(k)})_c \neq 0 \). Let \( c_{k+1} \) be the first oriented 3-cell (with respect to the ordering of the 3-cells) for which \( (d\nu^{(k)})_{c_{k+1}} \neq 0\). 
  Since \( \nu \in \Sigma_{P_B} \), we have \( (d\nu)_{c_{k+1}} = 0 \), and consequently there must be at least one plaquette \( p\in \support \nu \cap \bigl( \partial c_{k+1} \smallsetminus  \support \nu^{(k)} \bigr)\). Let \( p_{k+1}  \) be the first such plaquette (with respect to the ordering of the plaquettes).
  Define
        \begin{equation*}
            \nu^{(k+1)}_p \coloneqq \begin{cases}
            \nu_p &\text{if } p = \pm p_{k+1} \cr  
            \nu^{(k)}_p &\text{otherwise.}
            \end{cases}
        \end{equation*}
Note that if \( \nu^{(\ell)},\nu^{(\ell+1)}, \ldots, \nu^{(k)} \)   satisfies (i) and (ii), then so does \( \nu^{(k+1)} \). 
Using induction, we obtain a sequence \( \nu^{(\ell)}, \nu^{(\ell+1)},\ldots, \nu^{(m)} \) of 2-forms with \( \support \nu^{(\ell)} = P \) which satisfies~\ref{enum: no of configs proof i} and~\ref{enum: no of configs proof ii}. We now show that such a sequence must satisfy~\( \nu^{(m)} = \nu \). To this end, note that by~\ref{enum: no of configs proof i}, \( |\support \nu^{(m)} | = 2m \) and by~\ref{enum: no of configs proof ii}, \( \nu|_{\support \nu^{(m)}} = \nu^{(m)} \). Since \( |\support \nu| = 2m \), it follows that \( \nu^{(m)} = \nu \).

We now give an upper bound of the total number of sequences \( (p_{\ell+1},p_{\ell+2},\ldots, p_m ) \) which correspond, as above, to some  \( \nu \in \Pi_{P,m} \).
To obtain such an upper bound, note first that for each  \( k \in \{ \ell, \ell+1, \cdots, m-1 \} \), given the 3-cell \( c_{k+1} \) there are at most five possible choices for \( p_{k+1} \).
Consequently, the total number sequences \( (p_{\ell+1}, \ldots, p_m) \) which can correspond to some \( \nu \) as above is at most \( 5^{m-\ell} \).
Next, note that for each \( k \in \{ 1,2, \ldots, m \} \), we have \( \nu_{p_k} \in G\smallsetminus \{ 0 \} \).
Since for \( \nu \in \Pi_{P,m} \) the mapping \( \nu \mapsto (\nu^{(\ell)}, \nu^{(\ell+1)}, \ldots, \nu^{(m)}) \) 
    is injective, we obtain
\begin{equation*}
    \sum_{\nu \in \Pi_{P,m}} \phi_\beta  (\nu) 
    = 
    \sum_{\nu \in \Pi_{P,m}} \prod_{p \in \support \nu} \phi_\beta  (\nu_p) 
    \leq 
    5^{m-|P^+|} \biggl[ \, \sum_{g \in G\smallsetminus \{ 0 \}}\phi(g)^2 \biggr]^m.
\end{equation*}
Recalling the definition of \( \alpha(\beta) \), the desired conclusion follows. 
\end{proof}

\begin{proof}[Proof of Proposition~\ref{proposition: vortex probability}]
    For each \( m \geq |P^+| \) let \( \Pi_{P,m} \) be defined as in Lemma~\ref{lemma: counting vortex configurations}. Then, by definition, 
    \begin{equation*}
        \Pi_{P,M}^\geq = \bigcup_{m \geq M} \Pi_{P,m}.
    \end{equation*}
    Consequently, by a union bound, we have 
    \begin{equation*} 
        \begin{split}
            &\mu_{B,\beta}\bigl( \Pi_{P,M}^\geq\bigr)
        \leq 
        \sum_{m=M}^\infty
        \mu_{B,\beta}\bigl( \Pi_{P,m}\bigr)
        \leq
        \sum_{m=M}^\infty \sum_{\nu \in \Pi_{P,m}}
        \mu_{B,\beta} \bigl( \{ \omega \in \Sigma_{P_B} \colon \omega|_{\support \nu} = \nu \} \bigr).
        \end{split}
    \end{equation*}
    By Lemma~\ref{lemma: agreement probability}, for any \( m \geq M \) and any \( \nu \in \Pi_{P,m} \subseteq \Sigma_{P_B} \), we have
	\begin{equation*} 
        \mu_{B,\beta} \bigl( \{ \omega \in \Sigma_{P_B} \colon \omega|_{\support \nu} = \nu \} \bigr) 
	        \leq   \phi_\beta  (\nu) .
	\end{equation*}  
    Applying Lemma~\ref{lemma: counting vortex configurations}, we thus obtain
    \begin{equation*} 
        \begin{split}
            & 
        \sum_{m=M}^\infty \sum_{\nu \in \Pi_{P,m}}
        \mu_{B,\beta} \bigl( \{ \omega \in \Sigma_{P_B} \colon \omega|_{\support \nu} = \nu \} \bigr)
        \leq \sum_{m=M}^\infty 5^{m-|P^+|} \alpha(\beta)^m.
            \end{split}
    \end{equation*}
	The right-hand side in the previous equation is a geometric sum, which converges exactly if \({5\alpha(\beta)< 1.} \)
    In this case, by combining the previous equations, we obtain
    \begin{align*}
        \mu_{B,\beta}\bigl( \Pi_{P,M}^\geq\bigr)
        \leq
        \frac{5^{M-|P^+|} \alpha(\beta)^{M}}{1 - 5\alpha(\beta) }
    \end{align*}
    as desired.
\end{proof}

The next result  is a small variation of Proposition~\ref{proposition: vortex probability} which turns out to be useful when we work with pairs of plaquette configurations.

\begin{proposition} \label{proposition: vortex probability 2} 
    Let \( B \) and \( B' \) be two boxes in \( \mathbb{Z}^4 \) with \( B' \subseteq B \),  let  \( \beta \geq 0 \) be such that \( 5\alpha(\beta)< 1 \), and let \( P_0 \subseteq  P_{B'}\) be non-empty and symmetric. For \( M \geq |P_0^+| \), let \( \bar \Pi_{P_0, M}^\geq \) be the set of all pairs \( (\omega,\omega')  \in \Sigma_{P_B} \times \Sigma_{P_{B'}} \)  such that there is \( \nu \in \Sigma_{P_B} \) and \( \nu' \in \Sigma_{P_{B'}} \) with
    \begin{enumerate}[label=(\arabic*), series=vortexfliplemma]
        \item \( \omega|_{\support \nu} = \nu \) and \( \omega'|_{\support \nu'} = \nu' \),  \label{enum: vortex probability i 2}
        \item \( P_0 \subseteq \support \nu \cup \support \nu' \), \label{enum: vortex probability ii 2}
        \item for all symmetric sets \(  P \subseteq \support \nu\) and \(  P' \subseteq \support \nu'\) which satisfies satisfies \( P \sqcup P' = P_0 \),  \( \nu \) is \( P \)-irreducible and \( \nu' \) is \( P' \)-irreducible.
        \label{enum: vortex probability iii 2}  
        \item \( |\support \nu | + |\support \nu' | \geq 2M  \). \label{enum: vortex probability iv 2}
    \end{enumerate}
    Then
    \begin{equation}\label{eq: vortex probability 2} 
        \mu_{B,\beta} \times \mu_{B',\beta}\bigl( \bar \Pi_{P_0,M}^\geq) \leq  \frac{  2^{|P_0^+|} (M-|P_0^+|) \, 5^{M-|P_0^+|} \alpha(\beta)^{M}}{(1 - 5\alpha(\beta))^2 } .
    \end{equation} 
\end{proposition}

\begin{proof}
    For non-empty and symmetric sets \(   P \subseteq P_B \), and \( m \geq | P^+| \), recall the definition of \( \Pi_{B,  P,m}^\geq \) from Proposition~\ref{proposition: vortex probability}. 
    For \( m_0 \geq 0 \), define
    \begin{equation*}
        \Pi_{B,\emptyset,m_0}^\geq \coloneqq 
        \begin{cases}
            \Sigma_{P_B} &\text{if } m_0 = 0,\cr 
            \emptyset &\text{else.}
        \end{cases}
    \end{equation*}
    Analogously, for  non-empty and symmetric sets \(   P \subseteq P_{B'} \), and \( m \geq | P^+| \), recall the definition of \( \Pi_{B',  P,m}^\geq \) from Proposition~\ref{proposition: vortex probability}, and for \( m_0 \geq 0 \), define
    \begin{equation*}
        \Pi_{B',\emptyset,m_0}^\geq \coloneqq 
        \begin{cases}
            \Sigma_{P_{B'}} &\text{if } m_0 = 0,\cr 
            \emptyset &\text{else.}
        \end{cases}
    \end{equation*}
    Then, by definition, we have
    \begin{equation*}
        \bar \Pi_{P_0,M}^\geq \subseteq \bigcup_{\substack{P \subseteq P_0 \mathrlap{\colon}\\ P = -P}} \bigcup_{m_0 = |P^+|}^{M-|P^+\smallsetminus P^+|}  \Pi_{B,P,m_0}^\geq \times \Pi_{B',P_0\smallsetminus P,M-m_0}^\geq.
    \end{equation*} 
    Consequently, by a union bound, we have
    \begin{equation}\label{eq: step 1}
        \mu_{B,\beta} \times \mu_{B',\beta} \bigl( \bar \Pi_{P_0,M}^\geq \bigr)
        \leq 
        \sum_{\substack{P \subseteq P_0 \mathrlap{\colon}\\ P = -P}} \sum_{m_0 = |P^+|}^{M-|P_0^+\smallsetminus P^+|}   
        \mu_{B,\beta} \bigl( \Pi_{B,P,m_0}^\geq \bigr) \,  \mu_{B',\beta} \bigl( \Pi_{B',P_0\smallsetminus P,M-m_0}^\geq \bigr).
    \end{equation} 
    If \( P = \emptyset \), then 
    \begin{equation*}
        \begin{split}
            &\sum_{m_0 = |P^+|}^{M-|P_0^+\smallsetminus P^+|}   
            \mu_{B,\beta} \bigl( \Pi_{B,P,m_0}^\geq \bigr) \,  
            \mu_{B',\beta} \bigl( \Pi_{B',P_0\smallsetminus P,M-m_0}^\geq \bigr)
            =
            \mu_{B,\beta} \bigl( \Pi_{B,\emptyset, 0}^\geq \bigr) \,  \mu_{B',\beta} \bigl( \Pi_{B',P_9 ,M}^\geq \bigr) 
            \\&\qquad =
            \mu_{B',\beta} \bigl( \Pi_{B',P_0 ,M}^\geq \bigr).
        \end{split}
    \end{equation*} 
    Similarly, if \( P = P_0 \), then 
    \begin{equation*}
        \begin{split}
            &\sum_{m_0 = |P^+|}^{M-|P_0^+\smallsetminus P^+|}   
            \mu_{B,\beta} \bigl( \Pi_{B,P,m_0}^\geq \bigr) \,  
            \mu_{B',\beta} \bigl( \Pi_{B',P_0\smallsetminus P,M-m_0}^\geq \bigr)
            = 
            \mu_{B,\beta} \bigl( \Pi_{B,P_0,M}^\geq \bigr) \,  
            \mu_{B',\beta} \bigl( \Pi_{B',\emptyset,0}^\geq \bigr) 
            \\&\qquad=
            \mu_{B,\beta} \bigl( \Pi_{B,P_0,M}^\geq \bigr).
        \end{split}
    \end{equation*} 
    Consequently, we have
    \begin{equation}\label{eq: step 2}
        \begin{split} 
            &\sum_{\substack{P \subseteq P_0 \mathrlap{\colon}\\ P = -P}} \sum_{m_0 = |P^+|}^{M-|P_0^+\smallsetminus P^+|}   
            \mu_{B,\beta} \bigl( \Pi_{B,P,m_0}^\geq \bigr) \,  
            \mu_{B',\beta} \bigl( \Pi_{B',P_0\smallsetminus P,M-m_0}^\geq \bigr)
            \\&\qquad =
            \mu_{B,\beta} \bigl( \Pi_{B,P_0,M}^\geq \bigr)
            +
            \mu_{B',\beta} \bigl( \Pi_{B',P_0,M}^\geq \bigr)
            \\&\qquad\qquad+
            \sum_{\substack{\emptyset \subsetneq P \subsetneq P_0 \mathrlap{\colon}\\  P = -P}} \sum_{m_0 = |P^+|}^{M-|P_0^+\smallsetminus P^+|}   
            \mu_{B,\beta} \bigl( \Pi_{B,P,m_0}^\geq \bigr) \,  
            \mu_{B',\beta} \bigl( \Pi_{B',P_0\smallsetminus P,M-m_0}^\geq \bigr).
        \end{split}
    \end{equation}  
    Since \( 5\alpha(\beta)< 1 \), we can now apply Proposition~\ref{proposition: vortex probability} to obtain
    \begin{equation}\label{eq: step 3}
        \begin{split}
            &\mu_{B,\beta} \bigl( \Pi_{B,P_0,M}^\geq \bigr)
            +
            \mu_{B',\beta} \bigl( \Pi_{B',P_0,M}^\geq \bigr)
            +
            \sum_{\substack{\emptyset \subsetneq P \subsetneq P_0 \mathrlap{\colon}\\  P = -P}} \sum_{m_0 = |P^+|}^{M-|P_0^+\smallsetminus P^+|}   
            \mu_{B,\beta} \bigl( \Pi_{B,P,m_0}^\geq \bigr) \,  \mu_{B',\beta} \bigl( \Pi_{B',P_0\smallsetminus P,M-m_0}^\geq \bigr)
            \\&\qquad\leq
            2\frac{5^{M-|P_0^+|} \alpha(\beta)^{M}}{1 - 5\alpha(\beta) } 
            +
            \sum_{\substack{\emptyset \subsetneq P \subsetneq P_0 \mathrlap{\colon}\\  P = -P}} \sum_{m_0 = |P^+|}^{M-|P_0^+\smallsetminus P^+|}
            \frac{5^{m_0-|P^+|} \alpha(\beta)^{m_0}}{1 - 5\alpha(\beta) }
            \, \frac{5^{(M-m_0)-|P_0^+ \smallsetminus P^+|} \alpha(\beta)^{M-m_0}}{1 - 5\alpha(\beta) }
            \\&\qquad=
            2\frac{5^{M-|P^+|} \alpha(\beta)^{M}}{1 - 5\alpha(\beta) } 
            +
            \sum_{\substack{\emptyset \subsetneq P \subsetneq P_0 \mathrlap{\colon}\\  P = -P}} \sum_{m_0 = |P^+|}^{M-|P_0^+\smallsetminus P^+|}
            \frac{5^{M-|P_0^+|} \alpha(\beta)^{M}}{(1 - 5\alpha(\beta))^2 } 
            \\&\qquad=
            2\frac{5^{M-|P_0^+|} \alpha(\beta)^{M}}{1 - 5\alpha(\beta) } 
            + (2^{|P_0^+|}-2) (M-|P_0^+|+1)
            \frac{5^{M-|P_0^+|} \alpha(\beta)^{M}}{(1 - 5\alpha(\beta))^2 } 
            \\&\qquad\leq
            2^{|P_0^+|} (M-|P_0^+|+1)
            \frac{5^{M-|P_0^+|} \alpha(\beta)^{M}}{(1 - 5\alpha(\beta))^2 }.
        \end{split}
    \end{equation}
    Combining~\eqref{eq: step 1},~\eqref{eq: step 2} and~\eqref{eq: step 3}, we finally obtain~\eqref{eq: vortex probability 2} as desired.
\end{proof}

The last result of this section is the following proposition, which provides a matching lower bound for the inequality in Lemma~\ref{lemma: agreement probability}.

\begin{proposition}\label{proposition: lower bound of containing a specific vortex}
    Let \( B \) be a box in \( \mathbb{Z}^4 \), let \( \beta >0 \) be such that \( 5\alpha(\beta)< 1 \), and let \( \nu \in \Sigma_{P_B} \) be such that \( \dist_B(\support \nu, \delta B) \geq 7 \).
	Then 
	\begin{equation}\label{eq: lower bound of containing a specific vortex}
	     \biggl(1 -  \frac{5^{5} \alpha(\beta)^{6}|\support \nu | }{2(1 - 5\alpha(\beta)) } \biggr) \, \phi_\beta  (\nu)  \leq \mu_{B,\beta} \bigl( \{ \omega \in \Sigma_{P_B} \colon \omega|_{\support \nu} = \nu \} \bigr) \leq   \phi_\beta  (\nu).
	\end{equation} 
\end{proposition}

\begin{remark}
    For any given \( \nu \in \Sigma_{P_B} \), he ratio between the upper and lower bound in~\eqref{eq: lower bound of containing a specific vortex} can be made arbitrarily close to one by taking \( \beta \) sufficiently large.
\end{remark}

\begin{proof}
    Since the upper bound in~\eqref{eq: lower bound of containing a specific vortex} is a direct consequence of Lemma~\ref{lemma: agreement probability}, we only need to show that the lower bound in~\eqref{eq: lower bound of containing a specific vortex} holds. To this end, note first that 
	\begin{equation}\label{eq: product}
	    \begin{split}
	    &\mu_{B,\beta} \bigl( \{ \omega \in \Sigma_{P_B} \colon \omega|_{\support \nu} = \nu \} \bigr) 
	    \\&\qquad=
	    \frac{
	    \mu_{B,\beta} \bigl( \{ \omega \in \Sigma_{P_B} \colon \omega|_{\support \nu} = \nu \} \bigr) }{
	    \mu_{B,\beta} \bigl( \{ \omega \in \Sigma_{P_B} \colon \omega|_{\support \nu} = 0 \} \bigr) } \cdot 
	    \mu_{B,\beta} \bigl( \{ \omega \in \Sigma_{P_B} \colon \omega|_{\support \nu} = 0 \} \bigr),
	    \end{split}
	\end{equation}

	Assume that \( p \in P_B \) and that  \( \omega' \in \Sigma_{P_B} \) is such that \( \omega'|_p \neq 0 \). By~Lemma~\ref{lemma: irreducible form relative set} there exists an irreducible and non-trivial plaquette configuration \( \nu' \in \Sigma_{P_B} \) with \( p \in \support \nu' \) and \( \omega'|_{\support \nu'} = \nu' \). By Lemma~\ref{lemma: small vortices}, any non-trivial plaquette configuration \( \nu' \in \Sigma_{P_B} \) must satisfy \( |(\support \nu')^+ | \geq 6 \). Consequently, by Proposition~\ref{proposition: vortex probability}, applied with \( P \coloneqq \{ p,-p \} \) and \( M = 6 \), we have 
    \begin{equation*}
         \mu_{B,\beta} \bigl( \{ \omega \in \Sigma_{P_B} \colon \omega|_p \neq 0 \} \bigr) \leq   \frac{5^{6-1} \alpha(\beta)^{6}}{1 - 5\alpha(\beta) } .
    \end{equation*}
    Next note that for \( \omega \in \Sigma_{P_B} \), the event \( \omega|_{\support \nu} \neq 0 \) is equivalent to that \( \support \omega \cap (\support \nu)^+ \neq \emptyset \).
    Consequently, by a union bound, we obtain
    \begin{equation}\label{eq: the lower bound}
         \mu_{B,\beta} \bigl( \{ \omega \in \Sigma_{P_N} \colon \omega|_{\support \nu} = 0 \} \bigr) \geq 1 - \sum_{p \in (\support \nu)^+} \frac{5^{6-1} \alpha(\beta)^{6}}{1 - 5\alpha(\beta) } .
    \end{equation}
    On the other hand, by Lemma~\ref{lemma: ratio}, we have
	\begin{equation}\label{eq: first ratio}
	    \frac{\mu_{B,\beta} \bigl( \{ \omega \in \Sigma_{P_B} \colon \omega|_{\support \nu} = \nu \} \bigr)}{\mu_{B,\beta} \bigl( \{ \omega \in \Sigma_{P_B} \colon \omega|_{\support \nu} = 0 \} \bigr)} =   \phi_\beta  (\nu).
	\end{equation} 
    Combining~\eqref{eq: product},~\eqref{eq: the lower bound} and~\eqref{eq: first ratio}, we obtain the desired lower bound.
\end{proof}

\section{Paths, optimal paths and geodesics}\label{section: paths}

The following definition, which we recall from the introduction, will be central in the rest of this paper.
\begin{definition}\label{definition: graph definition}
    Given two boxes \( B \) and \( B' \) in \( \mathbb{Z}^4 \),  \( \omega  \in \Sigma_{P_B} \) and \(  \omega' \in \Sigma_{P_{B'}} \), let \( \mathcal{G}(\omega,\omega' ) \) be the graph with vertex set \( \support \omega \cup \support \omega' \) and an edge between two distinct plaquettes \( p_1,p_2 \in  \support \omega \cup \support \omega' \) if either \( \hat \partial p_1 \cap \hat \partial p_2 \neq \emptyset  \) or \( \hat \partial p_1 \cap \hat \partial(-p_2) \neq \emptyset  \). 
\end{definition}

If \( p_1,p_2 \in \support \omega \cup \support \omega'  \) are neighbors in \( \mathcal{G}(\omega,\omega' ) \), we write \( p_1 \sim p_2 \).

The main reason for introducing the graph \(  \mathcal{G}(\omega,\omega' ) \) is to be able to talk about paths and geodesics in this graph. To be able to talk about general properties of such paths, we introduce the following notation.

\begin{definition}\label{def: path}
    Let \( B \) be a box in \( \mathbb{Z}^4 \), and let the plaquettes \(  p_1,p_2, \ldots, p_m \in P_B \) be distinct. If for each \( k \in \{ 1,2, \ldots, m-1 \}\) we have  \( \hat \partial p_k \cap \pm \hat \partial p_{k+1} \neq \emptyset \),
    then \( \mathcal{P} \coloneqq (p_1, p_2, \ldots, p_m) \) is said to be a \emph{path} (from \( p_1 \) to \( p_m \)) in \( P_B \).
\end{definition}

When \( \mathcal{P} = (p_1,p_2, \ldots, p_m )\) is a path, we let \( \image (\mathcal{P}) \) denote the set \( \{ p_1,p_2, \ldots, p_m \} \), and let \( |\mathcal{P}| \coloneqq |\image (\mathcal{P})| = m \).

Note that if \( B \) and \( B' \) are two boxes in \( \mathbb{Z}^4 \) with \( B' \subseteq B \) and \( \mathcal{P} \) is a path in \( P_{B'} \), then \( \mathcal{P} \) is a path in \( P_{B'} \).
Also, if \( \omega \in \Sigma_{P_B} \), \(  \omega'\in \Sigma_{P_{B'}} \), and \( \mathcal{P} \) is a path in \( \mathcal{G}(\omega,\omega') \), then \( \mathcal{P} \) is a path in  \( P_B \). Conversely, if \( \mathcal{P} \) is a path in \( P_B \), \( g \in G \smallsetminus \{ 0 \} \),  and we define \( \omega \in \Sigma_{P_B}\) and \( \omega' \in \Sigma_{P_{B'}}\) by letting \( \omega' \equiv 0 \) and \(  \omega_p = g \) for all \( p \in P_B^+ \), then \( \mathcal{P} \) is a path in \( \mathcal{G}(\omega,\omega') \). Note that in this example, \( \omega \) has full support. The next lemma shows that in some cases, there is \( \omega  \in \Sigma_{P_B}\) and \( \omega'\in \Sigma_{P_{B'}}\) with relatively small supports such that a given path \( \mathcal{P} \) is a path in \( \mathcal{G}(\omega,\omega') \).

\begin{lemma}\label{lemma: small configuration containing a path}
    Let \( B \) be a box in \( \mathbb{Z}^4 \), and let \( \mathcal{P} = (p_1,p_2,\ldots, p_m) \) be a shortest path from \( p_1 \) to \( p_m \)  in \( P_B \). Then there is \( \omega^{(0)},\omega^{(1)} \in \Sigma_{P_B} \) such that 
    \begin{enumerate}[label=(\arabic*)]
        \item \( \mathcal{P} \) is a path in \( \mathcal{G} ( \omega^{(0)} , \omega^{(1)}) \), and \label{item: confix with path i}
        \item \( |\support \omega^{(0)}| + |\support \omega^{(1)}| \leq 12|\mathcal{P}| \).\label{item: confix with path ii} 
    \end{enumerate} 
\end{lemma}

\begin{proof}
      
    We first choose a sequence of oriented edges as follows.
    \begin{enumerate}[label=(\roman*)]
        \item Let \( k_1 \coloneqq  1 \). If \( m \geq 2 \) and \( \partial p_1 \cap \pm \partial p_2 \neq \emptyset \), pick \( e_1 \in \partial p_1 \cap \pm \partial p_2 \). Otherwise, pick any \( e_1 \in \partial p_1 \).
    \end{enumerate} 
    Now assume that \( e_1, e_2,\ldots, e_{j-1} \) are given for some \( j \geq 2 \). If there is \( k \in \{ 1, \ldots, m \} \) such that  \(  \partial p_k \cap \pm \{ e_1,e_2, \ldots, e_{j-1} \} = \emptyset \), we choose \( e_{j} \) as follows.
    \begin{enumerate}[label=(\roman*), resume]
        \item Let \(  k_{j} \geq 1 \) be the smallest integer such that \( \partial p_{k_j} \cap \pm \{ e_1, e_2, \ldots, e_{j-1} \} = \emptyset \). 
    If \( k_{j}+1 \leq m \) and \( \partial p_{k_j} \cap \pm \partial p_{k_j+1} \neq \emptyset \), pick \( e_{j} \in \partial p_{k_j} \cap \pm \partial p_{k_j+1} \). Otherwise, pick any \( e_{j} \in \partial p_{k_j} \). 
    \end{enumerate}
    Let \( \ell \) be the smallest positive integer which is such that
    \(  \partial p_k \cap \pm \{ e_1,e_2, \ldots, e_\ell \} \neq \emptyset \)
    for all \( k \in \{ 1,2,\ldots, m\} \).
    
    Fix any \( g \in G \smallsetminus \{ 0 \} \) and define \( \sigma^{(0)},\sigma^{(1)} \in \Sigma_{E_N} \) by 
    \begin{equation*}
        \sigma^{(0)}_e = \sum_{\substack{j \in \{ 1,2, \ldots, \ell \} \mathrlap{\colon} \\ j\text{ is even}}} \bigl( g \mathbb{1}_{e= e_j}-g \mathbb{1}_{e= -e_j} \bigr),\quad e \in E_B
    \end{equation*}
    and 
    \begin{equation*}
        \sigma^{(1)}_e = \sum_{\substack{j \in \{ 1,2, \ldots, \ell \} \mathrlap{\colon}\\ j\text{ is odd}}}\bigl( g \mathbb{1}_{e= e_j}-g \mathbb{1}_{e= -e_j} \bigr),\quad e \in E_B.
    \end{equation*}
    We will show that \( \omega^{(0)} \coloneqq d\sigma^{(0)} \) and \( \omega^{(1)} \coloneqq d\sigma^{(1)} \) have the desired properties.

    Since \( \mathcal{P} \) is a path, \ref{item: confix with path i} is equivalent to that 
    \begin{equation}\label{eq: path in support}
        \{ p_1,p_2,\ldots, p_m \} \subseteq \support \omega^{(0)} \cup \support \omega^{(1)}.
    \end{equation}
    To see that this holds, fix some \( k \in \{ 1,2, \ldots, m \} \). Define
    \begin{equation*}
        J_k \coloneqq \bigl\{ j \in \{ 1,2, \ldots, \ell \} \colon e_j \in \pm \partial p_k  \bigr\}.
    \end{equation*}
    By the choice of the edges \( e_1,e_2,\ldots,e_\ell \), the set \( J_k \) is non-empty. 
    Assume first that \( |J_k| \geq 2 \), and let \( j,j' \in J_k \) be such that \( j<j' \). 
    By the choice of \( e_j \), we cannot have \( p_{k_j} \sim p_{k_{j'}} \). Since \( (p_{k_j},p_k,p_{k_{j'}}) \) is a path in \( \mathcal{G}(\omega^{(0)},\omega^{(1)})\) and \( e_j \in \pm \partial p \), we must have \( |j-j'|=1 \). Since \( j' > j \), we have \( j' = j+ 1 \). 
    To sum up, we have either \( J_k = \{ j \} \) for some \( j \in \{ 1,2, \ldots, \ell \} \), or \( J_k = \{ j,j+1 \} \) for some \( j \in \{ 1,2, \ldots, \ell-1 \} \).
    Consequently, we have
    \begin{equation*}
        p_k \in \begin{cases}
            \omega^{(0)} &\text{if } j \text{ is even} \cr
            \omega^{(1)} &\text{if } j \text{ is odd,}
        \end{cases}
    \end{equation*}
    and thus~\ref{item: confix with path i} holds.
    To see that~\ref{item: confix with path ii} holds, note simply that 
    \begin{equation*}
        |\support \omega^{(0)}| + |\support \omega^{(1)}| \leq 6\bigl(|\support \sigma^{(0)}| + |\support \sigma^{(1)}|\bigr) = 12\ell \leq 12|\mathcal{P}|.
    \end{equation*} 
    This concludes the proof.
\end{proof}

\begin{lemma}\label{lemma: upper path bound}
    Let \( B \) and \( B' \) be two boxes in \( \mathbb{Z}^4 \),  let \( P_1,P_2 \subseteq P_B \cup P_B \) be disjoint, and assume that \( \omega \in \Sigma_{P_B} \) and \( \omega' \in \Sigma_{P_{B'}} \) are such that that \( P_1 \leftrightarrow P_2 \) in \( \mathcal{G}(\omega ,\omega') \). Further, let \( \mathcal{P} \coloneqq (p_1,p_2, \ldots, p_\ell)  \) be a geodesic in \( \mathcal{G}(\omega,\omega') \) between \( p_1 \in P_1 \) and \( p_\ell \in P_2 \). Then
    \begin{equation*}
        |\mathcal{P}|  \leq  \bigl( |\support \omega  | + | \support \omega'|\bigr)/2 .
    \end{equation*}
\end{lemma}

\begin{proof}
    Since \( \mathcal{P} \) is a geodesic in \( \mathcal{G} \), we must have 
    \begin{equation*}
        \image (\mathcal{P}) \cap \bigl(-\image(\mathcal{P})\bigr) = \emptyset,
    \end{equation*}
    and
    \begin{equation*}
        \pm \image(\mathcal{P}) \subseteq \support \omega \cup \support \omega'.
    \end{equation*}  
    Consequently, we have 
    \begin{equation}\label{eq: first path inequality}
        2|\mathcal{P}| \leq |\support \omega \cup \support \omega'| \leq |\support \omega | + |\support \omega'|.
    \end{equation}
    From this the desired conclusion immediately follows.
\end{proof}

\begin{remark}
    If we either assume that \(  \partial \hat \partial p \cap \delta P_B = \emptyset  \) for all \( p \in \image P \) or consider lattice gauge theory with zero boundary conditions, then one can quite easily show that~\eqref{eq: first path inequality} can be replaced with 
    \begin{equation*}
        10|\mathcal{P}|/3 \leq  |\support \omega| + | \support \omega'|.
    \end{equation*}
    If we have zero boundary conditions, then this stronger inequality can be used to replace the constant \( C_2 \) defined in~\eqref{eq: C2} with the smaller constant \( 3^{3/5} \cdot 10 \).
\end{remark}

\begin{definition}
    A path \( \mathcal{P} = ( p_1, p_2 \ldots, p_m ) \) is said to be \emph{optimal} if 
\begin{enumerate}[label=(\roman*)]
    \item \( p_1,p_2, \ldots,  p_{m-1} \) all have the same orientation, and
    \item for any \( i,j \in \{ 1,2, \ldots, m \} \), if \( \hat \partial p_i \cap (\hat \partial p_j \cup \hat \partial (-p_j)) \neq \emptyset \), then \( |i-j| \leq 1 \).
\end{enumerate} 
\end{definition}

The next lemma gives an upper bound on the number of optimal paths of a given length.

\begin{lemma}\label{lemma: number of optimal paths}
    Let \( B \) be a box in \( \mathbb{Z}^4 \), let \(  p_1 \in P_B \), and let \( m \geq 2 \). Then there are at most \( 40 \cdot 15^{m-2}  \) optimal paths \( \mathcal{P} = ( p_1,p_2,\ldots, p_m ) \) with  \( p_m \neq \pm  p_1 \).
\end{lemma}

\begin{proof} 
    Note first that for any \( p \in P_B \), we have \( |\hat \partial p| = 4 \), and if \( c \in \hat \partial p \), then \( |(\partial c)\smallsetminus \{ p \}| = 5 \). 
    
    Since \( \mathcal{P} \) is a path, we have 
    \begin{equation*}
        p_2 \in  \pm \bigl(\partial \hat \partial p_1 \smallsetminus \{ p_1 \} \bigr).
    \end{equation*}
    Since
    \begin{equation*}
        \Bigl|\pm \bigl(\partial \hat \partial p_1 \smallsetminus \{ p_1 \} \bigr) \Bigr| = 2 \cdot 4 \cdot 5 = 40,
    \end{equation*}
    the desired conclusion holds in the case \( m = 2 \). 
    On the other hand, for each plaquette \( p \in \pm \bigl(\partial \hat \partial p_1 \smallsetminus \{ p_1 \} \bigr) \), we also have \( -p \in \pm \bigl(\partial \hat \partial p_1 \smallsetminus \{ p_1 \} \bigr) \), and exactly one of \( p \) and \( -p \) has the same orientation as \( p_1 \). Since \( \mathcal{P} \) is optimal, if \( m \geq 3\) then the plaquettes \( p_1 \) and \( p_2 \) must have the same orientation, and hence in this case, given \( p_1 \) there are at most \( 40/2 = 20 \) possible choices of \( p_2 \).
    
    Next, assume that \( m \geq 4 \) and that \( j \in \{3, \ldots, m-1 \} \). Since \( \mathcal{P} \) is a path, we have
    \begin{equation*}
        p_{j} \in \pm \bigl( \partial \hat \partial p_{j-1} \smallsetminus \{ p_{j-1} \} \bigr),
    \end{equation*}
    and since \( \mathcal{P} \) is optimal, we have
    \begin{equation*}
        p_{j} \not \in \pm \partial \hat \partial p_{j-2}.
    \end{equation*}
    Since \( \mathcal{P} \) is a path, we have \( |\hat \partial p_{j-1} \cap \pm \hat \partial p_{j-2} | \geq 1 \). Since \( j<m \), the plaquettes \( p_j \) and \( p_{j-1} \) must have the same orientation, and hence, given \( p_1,p_2, \ldots, p_{j-1} \), there are at most \( 2 \cdot (4-1)\cdot 5 \cdot \frac12= 15 \) possible choices of \( p_j \). By the same argument, since the last plaquette can have any orientation, it follows that given \( p_1, p_2, \ldots,  p_{m-1} \),  there can be at most \( 2 \cdot (4-1)\cdot 5 = 30 \) possible choices of  \( p_m \). 
    
    To sum up, we have showed that there are exactly \( 40\)  optimal paths \( \mathcal{P} = (p_1,p_2) \) with \( p_2 \neq \pm p_1 \), and if \( m \geq 3 \), there are at most \( 20 \cdot 15^{m-3} \cdot 30\)  optimal paths \( \mathcal{P} = (p_1,p_2, \ldots, p_m) \) with \( p_m \neq \pm p_1 \). This concludes the proof.
\end{proof}

\begin{lemma}\label{lemma: optimal geodesic}
    Let \( B \) and \( B' \) be two boxes in \( \mathbb{Z}^4 \), and let \( p_1,p_2 \in P_B \) be disjoint. Further, let \( \omega  \in \Sigma_{P_B} \) and \( \omega' \in \Sigma_{P_{B'}} \) and assume that   \( \{ p_1 \} \leftrightarrow  \{ p_2\} \) in \( \mathcal{G}(\omega,\omega') \). Then there is an optimal geodesic \( \mathcal{P} \)  from \( p_1 \) to \( p_2 \) in \( \mathcal{G}(\omega,\omega') \).
\end{lemma}

\begin{proof}
    Since \( \{ p_1 \} \leftrightarrow  \{ p_2\} \) in \( \mathcal{G}(\omega ,\omega') \), there exist at least one geodesic from \( p_1 \) to \( p_2 \) in \( \mathcal{G}(\omega,\omega') \). Let \( (p^{(1)},p^{(2)},\ldots, p^{(m)}) \) be such a geodesic.
    Then, for any \( \tau_2,\ldots, \tau_{m-1} \in \{ -1,1 \} \), the path 
    \begin{equation*}
        \mathcal{P} \coloneqq \bigl( p^{(1)}, \tau_2 p^{(2)}, \ldots, \tau_{m-1}p^{(m-1)},p^{(m)} \bigr)
    \end{equation*}
    is also a geodesic from \( p_1 \) to \( p_2 \) in \( \mathcal{G}(\omega,\omega') \). In particular, we can choose \( \tau_2,\ldots, \tau_{m-1} \in \{ -1,1 \} \) such that \( \mathcal{P} \) is an optimal geodesic from \( p_1 \) to \( p_2 \) in \( \mathcal{G}(\omega,\omega') \). This concludes the proof.
\end{proof}

\section{Connected sets of plaquettes}\label{sec: connected sets}

\begin{definition}\label{definition: connected sets}
    Let \( B \) and \( B' \) be two boxes in  in \( \mathbb{Z}^4 \), and let \( \omega  \in \Sigma_{P_B} \) and \( \omega'\in \Sigma_{P_{B'}} \).
    Let \( P_1,P_2 \subseteq P_B \cup P_{B'} \) be two disjoint sets. If there is \( p_1 \in P_1 \) and \( p_2 \in P_2 \) and a path from \( p_1 \) to \( p_2 \) in \( \mathcal{G}(\omega,\omega') \), then we say that \( P_1 \) and \( P_2 \) are \emph{connected} in \( \mathcal{G}( \omega, \omega') \) and write  \( P_1  \leftrightarrow P_2 \) (in \( \mathcal{G}(\omega,\omega') \)). Otherwise, we say that \( P_1 \) and \( P_2 \) are \emph{disconnected}, and write \( P_1 \nleftrightarrow P_2 \) (in \( \mathcal{G}(\omega,\omega') \)). 
\end{definition}

Note that with this definition, if \( p_1,p_2 \in P_B \cup P_{B'} \), \( \omega  \in \Sigma_{P_B} \) and \( \omega' \in \Sigma_{P_{B'}} \), then there is a path from \( p_1 \) to \( p_2 \) in \( \mathcal{G} (\omega,\omega') \) if and only if \( \{ p_1\} \leftrightarrow \{ p_2\} \) in \( \mathcal{G}(\omega,\omega') \).
Note also that if \( P_1,P_2 \subseteq P_B \cup P_{B'} \), \( \omega \in \Sigma_{P_B} \) and \( \omega' \in \Sigma_{P_{B'}} \) then, by definition, \( P_1 \leftrightarrow P_2 \) in \( \mathcal{G}(\omega,\omega') \) if and only if  \( P_2 \leftrightarrow P_1 \) in \( \mathcal{G}(\omega,\omega') \). Moreover, \( P_1 \leftrightarrow P_2 \) in \( \mathcal{G}(\omega,\omega') \) if and only if \( \delta P_1 \leftrightarrow \delta P_2 \) in \( \mathcal{G}(\omega,\omega') \).

Re remark that if \( \omega \in \Sigma_{P_B} \), then the connected components of \( \mathcal{G}(\omega,0) \) are called \emph{vortices} in~\cite{c2018b} and~\cite{cao2020}, and correspond to the support of \emph{polymers} in~\cite{b1984}.

The main reason for introducing a notion of sets of plaquettes being connected is the following lemma.

\begin{lemma}\label{lemma: disconnected expectations}
    Let \( B \) be a box in \( \mathbb{Z}^4 \), and let \( \beta \geq 0 \). Further,  let  \( P_1,P_2 \subseteq P_{B}\) be symmetric and disjoint, and let \( f \colon \Sigma_{P_B} \to \mathbb{C} \) and \( g \colon \Sigma_{P_B} \to \mathbb{C} \) be such that \( f(\omega) = f(\omega|_{P_1}) \) and \( g(\omega) = g(\omega|_{P_2}) \) for all \( \omega \in \Sigma_{P_B} \)=. Finally,  let \( \omega^{(0)}, \omega^{(1)} \sim \mu_{B,\beta} \)   be independent. Then
    \begin{equation} \label{eq: connected set consequence}
        \begin{split}
            &\mathbb{E}_{B,\beta}  \times \mathbb{E}_{B,\beta} \Bigl[ f\bigl(\omega^{(0)}|_{P_1}\bigr)g \bigl(\omega^{(0)}|_{P_2}\bigr) \cdot \mathbb{1}\bigl[ P_1 \nleftrightarrow P_2 \text{ in } \mathcal{G}(\omega^{(0)},\omega^{(1)}) \bigr]\Bigr]
            \\&\qquad =
            \mathbb{E}_{B,\beta} \times \mathbb{E}_{B,\beta} \Bigl[ f\bigl(\omega^{(0)}|_{P_1}\bigr)g \bigl(\omega^{(1)}|_{P_2}\bigr) \cdot \mathbb{1}\bigl[ P_1 \nleftrightarrow P_2 \text{ in } \mathcal{G}(\omega^{(0)},\omega^{(1)})\bigr]\Bigr].
        \end{split}
    \end{equation} 
\end{lemma}

\begin{proof} 
    Given \( \omega^{(0)}, \omega^{(1)}\in \Sigma_{P_B} \), let 
    \begin{equation*} 
        {\hat P}_1 
        \coloneqq
        {\hat P}_1  (\omega,\omega')
        \coloneqq
        \bigl\{ p \in P_B \colon \partial ( \hat \partial p) \leftrightarrow P_1 \text{ in } \mathcal{G}( \omega, \omega')\bigr\}.
    \end{equation*} 
    Then \( \hat P_1 \) is the set of all plaquettes in \( P_B  \) which shares a 3-cell with some plaquette in \( P_B  \) which is connected to \( P_1 \). 
    Note that the set \(  {\hat {P_1}} \) can be determined without looking outside \(  {\hat P}_1 \) in \( \omega^{(0)} \) and \( \omega^{(1)} \), and that 
    \begin{equation*}
        d\bigl(\omega^{(0)}|_{\hat P_1} \bigr) = d\bigl(\omega^{(0)}|_{\hat P_1  } \bigr) = 0.
    \end{equation*}
    This implies that even if we know \( \omega^{(0)}|_{\hat P_1} \) and \( \omega^{(1)}|_{\hat P_1} \), we cannot distinguish between  \( \omega^{(0)}|_{P_{B} \smallsetminus \hat P_1} \) and \( \omega^{(0)}|_{P_{B} \smallsetminus \hat P_1} \). Consequently, we have
    \begin{align*}
        &\mathbb{E}_{B,\beta} \times \mathbb{E}_{B,\beta}   \Bigl[ g \bigl(\omega^{(0)}|_{P_2}\bigr)  \mid   P_2 \cap \hat P_1 = \emptyset,\, \omega^{(0)}|_{\hat P_1},\, \omega^{(1)}|_{\hat P_1} \Bigr]
        \\&\qquad=
        \mathbb{E}_{B,\beta}  \times \mathbb{E}_{B,\beta}   \Bigl[ g \bigl(\omega^{(1)}|_{P_2}\bigr)  \mid  P_2 \cap \hat P_1= \emptyset ,\, \omega^{(0)}|_{\hat P_1},\, \omega^{(1)}|_{\hat P_1} \Bigr].
    \end{align*}  
    This implies in particular that
    \begin{align*}  
        &
        \mathbb{E}_{B,\beta} \times \mathbb{E}_{B,\beta}  \Bigl[   f\bigl(\omega^{(0)}|_{P_1}\bigr)g \bigl(\omega^{(0)}|_{P_2}\bigr) \cdot \mathbb{1}\bigl[ P_2 \cap \hat P_1 = \emptyset \bigr]  \Bigr]  
        \\&\qquad= 
        \mathbb{E}_{B,\beta} \times \mathbb{E}_{B,\beta}  \Bigl[   f\bigl(\omega^{(0)}|_{P_1}\bigr)g \bigl(\omega^{(1)}|_{P_2}\bigr) \cdot \mathbb{1}\bigl[ P_2 \cap \hat P_1 = \emptyset \bigr]  \Bigr].
    \end{align*}  
    Noting that 
    \begin{equation*}
        P_1 \nleftrightarrow P_2 \text{ in } \mathcal{G}(\omega^{(0)}, \omega^{(1)})\Leftrightarrow P_2 \cap \hat P_1 = \emptyset,
    \end{equation*} 
    the desired conclusion follows. 
\end{proof}

\section{Distances between plaquettes}\label{sec: distances}

Assume that two boxes \( B \) and \( B' \) in \( \mathbb{Z}^4 \) with \( B' \subseteq B \), and sets \( P_1,P_2 \subseteq P_B \cup P_{B'} \) are given. Recall the definition of \( \dist_{B,B'}(P_1,P_2 ) \) from~\eqref{eq: definition of distance}. Using the notation of the previous section, we  have
\begin{equation*}
    \begin{split}
        \dist_{B,B'}(P_1,P_2) = \frac{1}{2} \min \Bigl\{ |\support \omega| + |\support \omega'| \colon \omega \in \Sigma_{P_B},\, \omega' \in \Sigma_{P_{B'}} \text{ s.t. } P_1 \leftrightarrow P_2 \text{ in } \mathcal{G}(\omega,\omega') \Bigr\}.
    \end{split}
\end{equation*}
We now define another measure of the distance between sets of plaquettes which will be useful to us.
\begin{equation*} 
    \dist_B^*(P_1,P_2) \coloneqq \min_{p_1 \in P_1,\, p_2 \in P_2} \min \bigl\{ |\mathcal{P}| \colon \mathcal{P} \text{ is a path from $p_1$ to $p_2$ in $P_B$} \bigr\}.
\end{equation*}
We mention that we do neither claim nor prove that the functions \( \dist_{B,B'} \) and \( \dist_B^* \) defined above are distance functions, and this will not be needed in the rest of this paper.

The following lemma gives a relationship between the distances \( \dist_{B,B'} \) and \( \dist_B^* \).

\begin{lemma}\label{lemma: distance comparison}
    Let \( B \) and \( B' \) be two boxes in \( \mathbb{Z}^4 \) with \( B' \subseteq B \), and let   \( P_1 ,P_2 \subseteq P_B \) be disjoint. Then
    \begin{equation}\label{eq: distance comparison}
    \dist_B^*(P_1,P_2) \leq \dist_{B,B'}(P_1,P_2).
\end{equation} 
and
    \begin{equation}\label{eq: distance comparison ii}
    \dist_{B}(P_1,P_2) \leq 6\dist_B^*(P_1,P_2).
\end{equation} 
\end{lemma}

\begin{proof} 
    We first prove that~\eqref{eq: distance comparison ii} holds. 
    To this end, let \( \mathcal{P} \) be a path from \( p_1 \in P_1 \) to \( p_2 \in P_2 \) in \( P_B \) with \( |\mathcal{P}| = \dist_B^*(P_1,P_2) \). Then \( \mathcal{P} \) is a shortest path from \( p_1 \) to \( p_2 \), and hence by Lemma~\ref{lemma: small configuration containing a path} there is \( \omega^{(0)},\omega^{(1)} \in \Sigma_{P_B} \) such that \( \mathcal{P} \) is a path in \( \mathcal{G}(\omega^{(0)},\omega^{(1)}) \) and \( |\support \omega^{(0)}| + |\support \omega^{(1)}| \leq 12|\mathcal{P}| \). Since  \( \mathcal{P} \) is a path in \( \mathcal{G}(\omega^{(0)},\omega^{(1)}) \) from \( p_1 \in P_1 \) to \( p_2 \in P_2 \), we have \( P_1 \leftrightarrow P_2 \) in \( \mathcal{G}(\omega^{(0)},\omega^{(1)}) \), and hence 
    \begin{equation*}
        \dist_B(P_1,P_2) \leq \frac{1}{2} \bigl( |\support \omega^{(0)}| + |\support \omega^{(1)}| \bigr) 
        \leq \frac{12 |\mathcal{P}|}{2} = 6|\mathcal{P}| = 6\dist_{B}^*(P_1,P_2).
    \end{equation*} 
    This concludes the proof of~\eqref{eq: distance comparison ii}.

    We now show that~\eqref{eq: distance comparison} holds. To this end, assume that \( \omega \in \Sigma_{P_B} \) and \( \omega' \in \Sigma_{P_{B'}} \) are such that that \( P_1 \leftrightarrow P_2 \) in \( \mathcal{G}(\omega,\omega') \) and \( |\support \omega | + | \support \omega'| = 2\dist_{B,B'}(P_1,P_2) \). By the definition of \( \dist_{B,B'}(P_1,P_2) \), such plaquette configurations \( \omega \) and \( \omega' \) exists.  
    Since \( P_1 \leftrightarrow P_2 \) in \( \mathcal{G}(\omega,\omega') \), there exists at least one path in \( \mathcal{G}(\omega,\omega') \) between some \( p_1 \in P_1 \) and some \( p_2 \in P_2 \). Let \( \mathcal{P} \) be such a path. Then by definition, we have
    \begin{equation*}
        \dist_{B,B'}^*(P_1,P_2) \leq |\mathcal{P}|.
    \end{equation*}
    Moreover, by Lemma~\ref{lemma: upper path bound}, we have
    \begin{equation*}
        |\mathcal{P}| \leq \bigl(|\support \omega| + |\support \omega'| \bigr)/2 = \dist_{B,B'}(P_1,P_2).
    \end{equation*}
    Combining these equations, we obtain
    \begin{equation*}
        \dist_B^*(P_1,P_2) \leq \dist_{B,B'}(P_1,P_2)
    \end{equation*}
    as desired.
\end{proof}

\section{The probability of two sets being connected}\label{sec: probability of being connected}

The purpose of this section is to give a proof of the following result, which gives an upper bound on the probability that two given sets are connected.

\begin{proposition}\label{proposition: probability of connected vs probability of path}
    Let \( B \) and \( B' \) be two boxes in \( \mathbb{Z}^4 \) with \( B' \subseteq B \), let \( \beta \geq 0 \) be such that \(  30 \alpha(\beta) < 1 \), and let \( P_1,P_2 \subseteq  P_{B } \) be disjoint. Then
    \begin{equation*}
        \begin{split}
            &\mu_{B,\beta}\times\mu_{B',\beta}\Bigl( \bigl\{ \omega  \in \Sigma_{P_B},\, \omega' \in \Sigma_{P_B} \colon P_1 \leftrightarrow P_2 \text{ in }  \mathcal{G}(\omega,\omega') \bigr\} \Bigr) 
            \leq 
            \frac{C_1|P_1|\bigl( C_2 \alpha(\beta) \bigr)^{  \dist_{B,B'}(P_1,P_2) } }{2}  . 
        \end{split}
    \end{equation*} 
\end{proposition}

Before giving a proof of this result at the end of this section, we will state and prove the following lemma, which gives an upper bound on the probability that a given path \( \mathcal{P} \) is a path in a random graph \( \mathcal{G}(\omega,\omega') \).

\begin{lemma}\label{lemma: geodesic probability}
    Let \( B \) and \( B' \) be two a boxes in \( \mathbb{Z}^4 \) with \( B' \subseteq B \), and let \( \beta \geq 0 \) be such that \(  5\alpha(\beta) < 1 \). Further, let \( p_1,p_2 \in P_{B} \) be distinct, and let \( \mathcal{P}  \) be a path from \( p_1 \) to \( p_2 \) in \( P_B \).
    Then
    \begin{equation*}
        \begin{split}
            &
            \mu_{B,\beta} \times \mu_{B',\beta} \bigl( \bigl\{ \omega \in   \Sigma_{P_B},\, \omega' \in   \Sigma_{P_{B'}} \colon \mathcal{P} \text{ is a geodesic in } \mathcal{G}(\omega,\omega') \bigr\} \bigr)  
            \leq  
             \frac{    5^{-|\mathcal{P}|} (10\alpha(\beta))^{M}}{2(1 - 5\alpha(\beta))^2 },
        \end{split}
    \end{equation*} 
    where \( M = \max\Bigl( |\mathcal{P}|, \dist_{B,B'}(p_1,p_2) \Bigr) \).
\end{lemma}

\begin{proof}
    By Lemma~\ref{lemma: upper path bound}, if \( \omega \in \Sigma_{P_B} \), \( \omega' \in \Sigma_{P_{B'}} \), and \( \mathcal{P} \) is a geodesic in \( \mathcal{G}(\omega, \omega') \), then 
    \begin{equation*}
        \frac{|\support \omega  |+| \support \omega'|}{2} \geq |\mathcal{P}|.
    \end{equation*} 
    On the other hand, by the definition of \( \dist_{B,B'} \), in this case we must also have 
    \begin{equation*}
        \frac{|\support \omega  |+| \support \omega'|}{2} \geq \dist_{B,B'}(p_1,p_2).
    \end{equation*} 
    Consequently,
    \begin{align*}
        &\mu_{B,\beta} \times \mu_{B',\beta} \bigl( \bigl\{ \omega\in   \Sigma_{P_B} ,\, \omega' \in   \Sigma_{P_{B'}} \colon \mathcal{P} \text{ is a geodesic from $p_1$ to $p_2$ in } \mathcal{G} (\omega,\omega') \bigr\} \bigr) 
        \\&\qquad
        \leq   
        \mu_{B,\beta} \times \mu_{B',\beta} \Bigl( \Bigl\{ \omega \in   \Sigma_{P_B},\, \omega' \in   \Sigma_{P_{B'}} \colon \mathcal{P} \text{ is a path from $p_1$ to $p_2$ in } \mathcal{G} (\omega,\omega') \text{ and } 
        \\&\hspace{12em}  \frac{|\support \omega | + |  \support \omega'|}{2} \geq \max\bigl(  |\mathcal{P}|, \dist_{B,B'}(p_1,p_2) \bigr) \Bigr\} \Bigr) .
    \end{align*} 
    Applying Proposition~\ref{proposition: vortex probability 2} with \( P \coloneqq \image\mathcal{P} \cup (- \image \mathcal{P}) \) and \( M \coloneqq \max\bigl(  |\mathcal{P}| , \dist_{B,B'}(p_1,p_2) \bigr) \), we thus obtain
    \begin{equation*}
        \begin{split}
            &
            \mu_{B,\beta} \times \mu_{B',\beta} \bigl( \bigl\{ \omega \in   \Sigma_{P_B} ,\, \omega' \in   \Sigma_{P_{B'}} \colon \mathcal{P} \text{ is a geodesic in } \mathcal{G}(\omega,\omega') \bigr\} \bigr) 
            \\&\qquad
            \leq  
             \frac{  2^{|\mathcal{P}|} (M-|\mathcal{P}|) \, 5^{M-|\mathcal{P}|} \alpha(\beta)^{M}}{(1 - 5\alpha(\beta))^2 }. 
        \end{split}
    \end{equation*} 
    Since \( p_1 \) and \( p_2 \) are distinct, we have \( |\mathcal{P}| \geq 2 \), and hence \( M -|\mathcal{P}| \geq 1. \) Using the inequality \( x \leq 2^{x-1} \), valid for all \( x \geq 1 \), 
    the desired conclusion follows.
\end{proof}

We now give a proof of Proposition~\ref{proposition: probability of connected vs probability of path}.

\begin{proof}[Proof of Proposition~\ref{proposition: probability of connected vs probability of path}]
    Let \( \omega \in\Sigma_{P_B} \) and \( \omega' \in\Sigma_{P_{B'}} \), and assume that \( P_1 \leftrightarrow P_2 \) in \( \mathcal{G}(\omega,\omega') \). Then, by Lemma~\ref{lemma: optimal geodesic}, there is \( p_1 \in P_1 \), \( p_2 \in P_2 \), and an optimal geodesic from \( p_1 \) to \( p_2 \) in \( \mathcal{G}(\omega, \omega') \).
    Consequently, we have the upper bound
    \begin{equation}\label{eq: connected to skeleton probability}
        \begin{split}
            &\mu_{B,\beta}\times\mu_{B',\beta}\Bigl( \bigl\{ \omega \in \Sigma_{P_B} ,\omega' \in \Sigma_{P_{B'}}   \colon P_1 \leftrightarrow P_2 \text{ in }  \mathcal{G}(\omega,\omega) \bigr\} \Bigr) 
            \\&\qquad \leq  
            \mu_{B,\beta}\times\mu_{B',\beta}\Bigl( \bigl\{ \omega \in \Sigma_{P_B} ,\, \omega' \in \Sigma_{P_{B'}}   \colon  \exists p_1 \in P_1,\, p_2 \in P_2 \text{ and }  \\[-.8ex]&\hspace{11em} \text{an optimal geodesic } \mathcal{P} \text{ from $p_1$ to $p_2$ in } \mathcal{G}(\omega,\omega')\bigr\}\Bigr) 
        \end{split}
    \end{equation}
    
    By Lemma~\ref{lemma: number of optimal paths}, given \( m \geq 2 \) and \( p_1 \in P_1 \), there are at most \(  40 \cdot 15^{m-2} \) optimal paths \( \mathcal{P} \) of length \( m \) which starts at \( p_1 \) and ends at some \( p \in P_B \). 
    Since \( P_1 \) and \( P_2 \) are disjoint, we have \( \dist_B^*(P_1,P_2)  \geq 2 \). 
    Consequently, for any \( m \geq \dist_B^*(P_1,P_2) \), there can be at most \( |P_1|\cdot 40 \cdot 15^{m-2} \) optimal paths which starts at some \( p_1 \in P_1\) and ends at some \( p_2 \in P_2 \).
    In particular, this implies that for any \( m \geq \dist_B^*(P_1,P_2) \), there are at most \( |P_1|\cdot 40 \cdot 15^{m-2} \) paths which starts at some \( p_1 \in P_1\) and ends at some \( p_2 \in P_2 \) which can be an optimal geodesic in \( \mathcal{G}(\omega,\omega') \) for some \( \omega \in \Sigma_{P_B} \) and \( \omega' \in \Sigma_{P_{B'}} \).
    On the other hand, given a path \( \mathcal{P} \) from \( p_1 \) to \( p_2 \), by Lemma~\ref{lemma: geodesic probability} we have
    \begin{equation*}
        \begin{split}
            &
            \mu_{B,\beta} \times \mu_{B,\beta} \bigl( \bigl\{ \omega \in   \Sigma_{P_B} ,\, \omega' \in   \Sigma_{P_{B'}} \colon \mathcal{P} \text{ is a geodesic in } \mathcal{G}(\omega,\omega') \bigr\} \bigr) 
            \\&\qquad
            \leq  
            \frac{5^{-|\mathcal{P}|} (10\alpha(\beta))^{\max ( |\mathcal{P}|, \dist_{B,B'}(p_1,p_2) )} }{2(1 - 5\alpha(\beta))^2 }. 
        \end{split}
    \end{equation*} 
    Combining there observations, and summing over all \( m \geq \dist_B^*(P_1,P_2) \), we thus obtain
    \begin{align} 
            &\mu_{B,\beta}\times\mu_{B',\beta}\Bigl( \bigl\{ \omega  \in \Sigma_{P_B}  ,\, \omega'\in \Sigma_{P_{B'}}   \colon  \exists p_1 \in P_1,\, p_2 \in P_2 \text{ and }  \\[-.8ex]&\hspace{8em} \text{an optimal geodesic } \mathcal{P}  \text{ from $p_1$ to $p_2$ in } \mathcal{G}(\omega,\omega')\bigr\}\Bigr)  \nonumber
            \\&\qquad\leq
            \sum_{m = \dist_B^*(P_1,P_2)}^\infty 40 |P_1| \cdot 15^{m-2}
            \cdot 
            \frac{5^{-m} \bigl( 10\alpha(\beta) \bigr)^{\max ( m, \dist_{B,B'}(P_1,P_2) )} }{2(1 - 5\alpha(\beta))^2 } \nonumber
            \\&\qquad=
            \frac{20 |P_1|}{15^2 (1 - 5\alpha(\beta))^2} \sum_{m = \dist_{B}^*(P_1,P_2)}^\infty  3^{m} \bigl( 10\alpha(\beta) \bigr)^{\max ( m, \dist_{B,B'}(P_1,P_2) )} 
             \label{eq: last line}
            .  
    \end{align}  
    We now rewrite the sum in~\eqref{eq: last line} as follows. First, note that by Lemma~\ref{lemma: distance comparison}, we have 
    \[
    \dist_{B}^*(P_1,P_2) \leq \dist_{B,B'}(P_1,P_2).
    \]
    Consequently,
    \begin{align} 
            &
            \sum_{m = \dist_B^*(P_1,P_2)}^\infty 
            3^{m} \bigl( 10 \alpha(\beta) \bigr)^{\max ( m, \dist_{B,B'}(p_1,p_2) )} \nonumber
            \\&\qquad= 
            \sum_{m = \dist_B^*(P_1,P_2)}^{\dist_{B,B'}(P_1,P_2) -1} 
            3^{m} \bigl( 10 \alpha(\beta) \bigr)^{  \dist_{B,B'}(P_1,P_2) }
            + 
            \sum_{m = \dist_{B,B'}(P_1,P_2)}^\infty 
            \bigl( 30 \alpha(\beta) \bigr)^{ m} 
            .  \label{eq: two sums}
    \end{align}  
    The first term in~\eqref{eq: two sums} is a finite geometric sum, and the second term is an infinite geometric sum which converge since \( 30\alpha(\beta)<1 \). Hence the previous equation is equal to

    \begin{equation*}
        \begin{split}
            & 
            \frac{3^{\dist_{B,B'}(P_1,P_2) } - 3^{\dist_B^*(P_1,P_2)}}{3-1}  \bigl( 10 \alpha(\beta) \bigr)^{  \dist_{B,B'}(P_1,P_2) } 
            + 
            \frac{\bigl(30 \alpha(\beta) \bigr)^{\dist_{B,B'}(P_1,P_2) } }{1-30 \alpha(\beta)   }.
        \end{split}
    \end{equation*}  
    Since \( 30 \alpha(\beta) < 1 \), we can bound the previous equation from above by
    \begin{equation*}
        \begin{split}
            & \biggl[
             \frac{1 }{2}  
            + 
            \frac{1 }{1-30 \alpha(\beta) }
            \biggr] \bigl( 30 \alpha(\beta) \bigr)^{  \dist_{B,B'}(P_1,P_2) }.
        \end{split}
    \end{equation*}
    Combining the previous equations, we obtain
    \begin{align*}
        &\mu_{B,\beta}\times\mu_{B',\beta}\Bigl( \bigl\{ \omega  \in \Sigma_{P_B},\,  \omega'  \in \Sigma_{P_{B'}}   \colon P_1 \leftrightarrow P_2 \text{ in }  (\omega,\omega') \bigr\} \Bigr) 
        \\&\qquad\leq  
        \frac{20 |P_1|}{15^2 (1 - 5\alpha(\beta))^2} \biggl[
             \frac{1 }{2}  
            + 
            \frac{1 }{1-30 \alpha(\beta)}
            \biggr] \bigl( 30 \alpha(\beta) \bigr)^{  \dist_{B,B'}(P_1,P_2). }
    \end{align*}
    Recalling the definition of \( C_1 \) from~\eqref{eq: C1} and the definition of \( C_2 \) from~\eqref{eq: C2}, we obtain the desired conclusion.
\end{proof}

\section{Covariance and connected sets}\label{sec: covariance}

We now use the notion of connected sets to obtain an upper bound on the covariance of local functions.

\begin{lemma}\label{lemma: covariance and chains}
    Let \( B \) be a box in \( \mathbb{Z}^4 \), and let \( \beta \geq 0 \). Further, let \( P_1 \subseteq P_B \) and \( P_2  \subseteq P_B \) be symmetric and disjoint, and let \( f_1,f_2 \colon \Sigma_{B,2} \to \mathbb{C} \) be such that \( f_1( \omega) = f_1( \omega|_{P_1}) \) and \( f_2( \omega) = f_2( \omega|_{P_2}) \) for all \(  \omega \in \Sigma_{P_B} \). Finally, let \( \omega, \omega' \sim \mu_{B,\beta} \) be independent. Then 
    \begin{equation*}
        \Bigl| \Cov \bigl(  f_1 ( \omega  ), f_2 (\omega) \bigr) \Bigr|
        \leq
        2\| f_1 \|_\infty \|  f_2 \|_\infty \, \mu_{B,\beta}\times\mu_{B,\beta} \Bigl( \bigl\{ \omega, \omega' \in  \Sigma_{P_B} \colon  P_1 \leftrightarrow P_2 \text{ in } \mathcal{G}(\omega,\omega') \bigr\} \Bigr).
    \end{equation*}
\end{lemma}

\begin{proof}
    Note first that
    \begin{equation*}
        \begin{split}
            \mathbb{E}_{B,\beta}\bigl[  f_1(\omega)  f_2(\omega )\bigr]
            &
            =
            \mathbb{E}_{B,\beta}\times \mathbb{E}_{B,\beta}\Bigl[ f_1(\omega)  f_2(\omega ) \cdot \mathbb{1} \bigl[ P_1 \leftrightarrow P_2 \text{ in } \mathcal{G}(\omega,\omega') \bigr] \Bigr] 
            \\&\qquad + 
            \mathbb{E}_{B,\beta} \times \mathbb{E}_{B,\beta}\Bigl[ f_1(\omega)  f_2(\omega )\cdot \mathbb{1} \bigl[ P_1 \nleftrightarrow P_2 \text{ in } \mathcal{G}(\omega,\omega') \bigr] \Bigr].
        \end{split}
    \end{equation*}
    Since \( f_1(\omega'') = f_1(\omega''|_{P_1}) \) and \( f_2(\omega'') = f_1(\omega''|_{P_2}) \) for all \(  \omega'' \in \Sigma_{P_B} \), we can apply Lemma~\ref{lemma: disconnected expectations} to obtain
    \begin{equation*}
        \begin{split}
            & 
            \mathbb{E}_{B,\beta}\times \mathbb{E}_{B,\beta} \Bigl[ f_1(\omega) f_2(\omega) \cdot \mathbb{1} \bigl[ P_1 \nleftrightarrow P_2 \text{ in } \mathcal{G}(\omega,\omega') \bigr] \Bigr]
            \\&\qquad 
            = 
            \mathbb{E}_{B,\beta}\times \mathbb{E}_{B,\beta} \Bigl[  f_1(\omega)f_2(\omega') \cdot \mathbb{1} \bigl[ P_1 \nleftrightarrow P_2 \text{ in } \mathcal{G}(\omega,\omega') \bigr] \Bigr]
            \\&\qquad
            = 
            \mathbb{E}_{B,\beta}\times \mathbb{E}_{B,\beta} \bigl[  f_1(\omega)  f_2(\omega')  \bigr]
            -
            \mathbb{E}_{B,\beta}\times \mathbb{E}_{B,\beta} \Bigl[ f_1(\omega) f_2(\omega') \cdot \mathbb{1} \bigl[ P_1 \leftrightarrow P_2 \text{ in } \mathcal{G}(\omega,\omega') \bigr] \Bigr]
            \\&\qquad
            = 
            \mathbb{E}_{B,\beta} \bigl[ f_1(\omega)  \bigr] \, 
            \mathbb{E}_{B,\beta} \bigl[  f_2(\omega')  \bigr]
            -
            \mathbb{E}_{B,\beta} \times \mathbb{E}_{B,\beta} \Bigl[ f_1(\omega) f_2(\omega') \cdot \mathbb{1} \bigl[ P_1 \leftrightarrow P_2 \text{ in } \mathcal{G}(\omega,\omega') \bigr] \Bigr].
        \end{split}
    \end{equation*}
    Combining the previous equations, we obtain
    \begin{equation*}
        \begin{split}
            &\mathbb{E}_{B,\beta} \bigl[  f_1(\omega)  f_2(\omega)\bigr]
            \\&\qquad
            =
            \mathbb{E}_{B,\beta} \bigl[ f_1(\omega)  \bigr] \, 
            \mathbb{E}_{B,\beta} \bigl[  f_2(\omega')  \bigr] 
            -
            \mathbb{E}_{B,\beta}\times \mathbb{E}_{B,\beta}\Bigl[  f_1(\omega)  f_2(\omega') \cdot \mathbb{1} \bigl[ P_1 \leftrightarrow P_2 \text{ in } \mathcal{G}(\omega,\omega') \bigr] \Bigr] 
            \\&\qquad\qquad
            +
            \mathbb{E}_{B,\beta}\times \mathbb{E}_{B,\beta} \Bigl[ f_1(\omega) f_2(\omega) \cdot \mathbb{1} \bigl[ P_1 \leftrightarrow P_2 \text{ in } \mathcal{G}(\omega,\omega') \bigr] \Bigr],
        \end{split}
    \end{equation*}
    and hence
    \begin{equation*}
        \begin{split}
            \Cov \Bigl(  f_1 ( \omega), f_2 (\omega) \Bigr)
            &=
            \mathbb{E}_{B,\beta} \times \mathbb{E}_{B,\beta}\Bigl[  f_1(\omega) f_2(\omega) \cdot \mathbb{1} \bigl[ P_1\leftrightarrow P_2 \text{ in } \mathcal{G}(\omega,\omega') \bigr] \Bigr]
            \\&\qquad
            -
            \mathbb{E}_{B,\beta} \times \mathbb{E}_{B,\beta}\Bigl[ f_1(\omega)f_2(\omega') \cdot \mathbb{1} \bigl[ P_1 \leftrightarrow P_2 \text{ in } \mathcal{G}(\omega,\omega') \bigr] \Bigr].
        \end{split}
    \end{equation*}
    From this the desired conclusion immediately follows.
\end{proof}

We now combine the previous lemma with Proposition~Proposition~\ref{proposition: probability of connected vs probability of path} to obtain our first main result.

\begin{proof}[Proof of Theorem~\ref{theorem: correlation length}]
    By combining Lemma~\ref{lemma: covariance and chains} and Proposition~\ref{proposition: probability of connected vs probability of path}, we immediately obtain
    \begin{equation*}
        \Bigl| \Cov \bigl(  f_1(d\sigma ), f_2 (d\sigma) \bigr)\Bigr| \leq 
            C_1 \bigl( C_2 \alpha(\beta)\bigr)^{ \dist_{B}(P_1,P_2) } \|f_1\|_\infty \, \|f_2 \|_\infty .
    \end{equation*}
    This concludes the proof of the first part of Theorem~\ref{theorem: correlation length}. 
    
    To see that the second claim of the theorem holds, we apply the first part of the theorem with \( f_1(\omega) = \tr \rho(\omega_{p_1}) \), \( f_2(\omega) = \tr \rho(\omega_{p_2}) \), \( P_1 = \{ p_1 \} \), and \( P_2 = \{ p_2 \} \), and note that since \( \rho \) is unitary, we have \( \| f_1 \|_\infty \leq \dim \rho \) and \( \| f_2 \|_\infty \leq \dim \rho \).
\end{proof}

\section{Total variation and connected sets}\label{sec: total variation}

\begin{lemma}\label{lemma: total variation bound}
    Let \( B' \subsetneq B\) be two boxes in \( \mathbb{Z}^4 \), and let \( \beta \geq 0 \). Let \( P_0 \subseteq P_{B'} \). Further, let \(  \omega  \sim \mu_{B,\beta} \) and \(  \omega'  \sim \mu_{B',\beta} \). Then
    \begin{equation*}
        \begin{split}
            &\dist_{TV} ( \omega|_{P_0},  \omega'|_{P_0}) 
            \leq  
            \mu_{B,\beta}\times \mu_{B',\beta}\Bigl( \bigl\{ \omega \in \Sigma_{P_B}, \omega' \in \Sigma_{P_{B'}} \colon
             P_0 \leftrightarrow (P_B\smallsetminus P_{B'})   \text{ in }   \mathcal{G}(\omega,\omega') \bigr\} \Bigr).
        \end{split}
    \end{equation*}
\end{lemma}
    
\begin{proof} 
    Given \( \omega  \in \Sigma_{P_B} \) and \( \omega' \in \Sigma_{P_{B'}} \), let  
    \begin{equation*}
        \hat P \coloneqq \{ p \in P_{B'} \colon \{ p\} \leftrightarrow P_B\smallsetminus P_{B'} \text{ in } \mathcal{G}(\omega ,\omega') \}
    \end{equation*}
    (see Figure~\ref{fig: b}).
    %
    Next, define  
    \begin{equation*}
        \hat{\hat P}
        \coloneqq
        (P_B\smallsetminus P_{B'}) \cup \bigl\{ p \in P_{B'} \colon \partial ( \hat \partial p) \leftrightarrow P_B \smallsetminus P_{B'} \text{ in } \mathcal{G}( \omega, \omega')\bigr\}.
    \end{equation*} 
    (see Figure~\ref{fig: c}) and note that \( P_B\smallsetminus P_{B'} \subseteq \hat{\hat P} \) and that
    \[
        \support \omega|_{\hat{\hat P}}  \cup \support \omega'|_{\hat{\hat P}}  = \hat P.
    \]  
    \begin{figure}[htb]
     \centering
     \begin{subfigure}[b]{0.3\textwidth}
        \centering
        \begin{tikzpicture}[scale=0.7] 
    
            \fill[fill = SpringGreen, fill opacity=0.3] (-3.3,1.6) circle (1.3mm); 
            \fill[fill = SpringGreen, fill opacity=0.3] (-3.5,-1.1) circle (1.4mm); 
            \fill[fill = SpringGreen, fill opacity=0.3] (-2.9,1.2) circle (2.5mm); 
            \fill[fill = SpringGreen, fill opacity=0.3] (-1.8,-0.4) circle (3.5mm);
            \fill[fill = SpringGreen, fill opacity=0.3] (-1,0.2) circle (4mm); 
            \fill[fill = SpringGreen, fill opacity=0.3] (-.55,1.8) circle (3mm);
            \fill[fill = SpringGreen, fill opacity=0.3] (-0.3,-0.9) circle (4mm);
            \fill[fill = SpringGreen, fill opacity=0.3] (0,2.5) circle (4mm);
            \fill[fill = SpringGreen, fill opacity=0.3] (0.35,-1.2) circle (2.5mm);
            \fill[fill = SpringGreen, fill opacity=0.3] (2,0.3) circle (4mm); 
    
            \fill[fill = SkyBlue, fill opacity=0.3] (-2.2,0.5) circle (2.5mm);
            \fill[fill = SkyBlue, fill opacity=0.3] (-1.3,-0.2) circle (3.5mm);
            \fill[fill = SkyBlue, fill opacity=0.3] (-0.4,-0.4) circle (3mm);  
            \fill[fill = SkyBlue, fill opacity=0.3] (0.0,1.2) circle (2.5mm);
            \fill[fill = SkyBlue, fill opacity=0.3] (-0.1,0.5) circle (2mm);
            \fill[fill = SkyBlue, fill opacity=0.3] (0.5,-0.1) circle (5mm);
            \fill[fill = SkyBlue, fill opacity=0.3] (0.8,1.0) circle (4.5mm);
            \fill[fill = SkyBlue, fill opacity=0.3] (1.3,-0.5) circle (3.1mm);  
    
            \draw[semithick] (-4,-1.5) -- (3,-1.5) -- (3,3) -- (-4,3) -- (-4,-1.5);
            \draw[] (-3,-1) -- (2,-1) -- (2,2.5) -- (-3,2.5) -- (-3,-1);
    
            \draw (-4,3) node[anchor=north west] {\small \( B \)};
            \draw (-3,2.5) node[anchor=north west] {\small \( B' \)};
    
            \draw[semithick] (-4,-1.5) -- (3,-1.5) -- (3,3) -- (-4,3) -- (-4,-1.5);
            \draw[] (-3,-1) -- (2,-1) -- (2,2.5) -- (-3,2.5) -- (-3,-1);
    
            \draw (-4,3) node[anchor=north west] {\small \( B \)};
            \draw (-3,2.5) node[anchor=north west] {\small \( B' \)};
            
            \end{tikzpicture}
            \caption{}\label{fig: a}
    \end{subfigure}
    \hfil
    \begin{subfigure}[b]{0.3\textwidth}
        \begin{tikzpicture}[scale=0.7]

            \fill[fill = SpringGreen, fill opacity=0.3] (-3.3,1.6) circle (1.3mm);
            \fill[fill = SpringGreen, fill opacity=0.3] (-3.5,-1.1) circle (1.4mm);
            \fill[fill = SpringGreen, fill opacity=0.3] (-1.8,-0.4) circle (3.5mm);
            \fill[fill = SpringGreen, fill opacity=0.3] (-1,0.2) circle (4mm); 
            \fill[fill = SpringGreen, fill opacity=0.3] (-.55,1.8) circle (3mm);
    
            \fill[fill = SkyBlue, fill opacity=0.3] (-2.2,0.5) circle (2.5mm);
            \fill[fill = SkyBlue, fill opacity=0.3] (-1.3,-0.2) circle (3.5mm);
            \fill[fill = SkyBlue, fill opacity=0.3] (0.0,1.2) circle (2.5mm);
            \fill[fill = SkyBlue, fill opacity=0.3] (-0.1,0.5) circle (2mm);
            \fill[fill = SkyBlue, fill opacity=0.3] (0.5,-0.1) circle (5mm);
            \fill[fill = SkyBlue, fill opacity=0.3] (0.8,1.0) circle (4.5mm);
            \fill[fill = SkyBlue, fill opacity=0.3] (1.3,-0.5) circle (3.1mm);
      
            \fill[fill = SpringGreen, fill opacity=0.3] (-2.9,1.2) circle (2.5mm); 
            \fill[fill = SpringGreen, fill opacity=0.3] (-0.3,-0.9) circle (4mm);
            \fill[fill = SpringGreen, fill opacity=0.3] (0,2.5) circle (4mm);
            \fill[fill = SpringGreen, fill opacity=0.3] (0.35,-1.2) circle (2.5mm);
            \fill[fill = SpringGreen, fill opacity=0.3] (2,0.3) circle (4mm);  
    
            \fill[fill = SkyBlue, fill opacity=0.4] (-0.4,-0.4) circle (3mm);   
    
            \begin{scope} 
                \clip (-2.9,1.2) circle (2.5mm) -- cycle
                    (-0.3,-0.9) circle (4mm) -- cycle
                    (0,2.5) circle (4mm) -- cycle
                    (0.35,-1.2) circle (2.5mm) -- cycle
                    (2,0.3) circle (4mm) -- cycle
                    (-0.4,-0.4) circle (3mm) -- cycle;
              
                \fill[red] (-3,-1) -- (-3,2.5) -- (2,2.5) -- (2,-1) -- (-3,-1); 
            \end{scope} 
    
            \draw[semithick] (-4,-1.5) -- (3,-1.5) -- (3,3) -- (-4,3) -- (-4,-1.5);
            \draw[] (-3,-1) -- (2,-1) -- (2,2.5) -- (-3,2.5) -- (-3,-1);
    
            \draw (-4,3) node[anchor=north west] {\small \( B \)};
            \draw (-3,2.5) node[anchor=north west] {\small \( B' \)};
            
        \end{tikzpicture}
        \caption{}\label{fig: b}
    \end{subfigure} 
    \hfil
    \begin{subfigure}[b]{0.3\textwidth}
        \begin{tikzpicture}[scale=0.7]
    
            \draw[semithick] (-4,-1.5) -- (3,-1.5) -- (3,3) -- (-4,3) -- (-4,-1.5);
            \draw[] (-3,-1) -- (2,-1) -- (2,2.5) -- (-3,2.5) -- (-3,-1);
    
            \draw (-4,3) node[anchor=north west] {\small \( B \)};
            \draw (-3,2.5) node[anchor=north west] {\small \( B' \)};

            \fill[fill = SpringGreen, fill opacity=0.3] (-3.3,1.6) circle (1.3mm);
            \fill[fill = SpringGreen, fill opacity=0.3] (-3.5,-1.1) circle (1.4mm);
            \fill[fill = SpringGreen, fill opacity=0.3] (-1.8,-0.4) circle (3.5mm);
            \fill[fill = SpringGreen, fill opacity=0.3] (-1,0.2) circle (4mm); 
            \fill[fill = SpringGreen, fill opacity=0.3] (-.55,1.8) circle (3mm);
    
            \fill[fill = SkyBlue, fill opacity=0.3] (-2.2,0.5) circle (2.5mm);
            \fill[fill = SkyBlue, fill opacity=0.3] (-1.3,-0.2) circle (3.5mm);
            \fill[fill = SkyBlue, fill opacity=0.3] (0.0,1.2) circle (2.5mm);
            \fill[fill = SkyBlue, fill opacity=0.3] (-0.1,0.5) circle (2mm);
            \fill[fill = SkyBlue, fill opacity=0.3] (0.5,-0.1) circle (5mm);
            \fill[fill = SkyBlue, fill opacity=0.3] (0.8,1.0) circle (4.5mm);
            \fill[fill = SkyBlue, fill opacity=0.3] (1.3,-0.5) circle (3.1mm);
    
            \fill[fill = SpringGreen, fill opacity=0.3] (-2.9,1.2) circle (2.5mm); 
            \fill[fill = SpringGreen, fill opacity=0.3] (-0.3,-0.9) circle (4mm);
            \fill[fill = SpringGreen, fill opacity=0.3] (0,2.5) circle (4mm);
            \fill[fill = SpringGreen, fill opacity=0.3] (0.35,-1.2) circle (2.5mm);
            \fill[fill = SpringGreen, fill opacity=0.3] (2,0.3) circle (4mm); 
            \fill[fill = SkyBlue, fill opacity=0.4] (-0.4,-0.4) circle (3mm);   
    
            \fill[fill = red]   
            (-4,-1.5) -- (3,-1.5) -- (3,3) -- (-4,3) -- cycle
            (-3,-1) -- (-3,2.5) -- (2,2.5) -- (2,-1) -- cycle
            (-2.9,1.2) circle (2.5mm) -- cycle
            (-0.3,-0.9) circle (4mm) -- cycle
            (0,2.5) circle (4mm) -- cycle
            (0.35,-1.2) circle (2.5mm) -- cycle
            (2,0.3) circle (4mm) -- cycle
            (-0.4,-0.4) circle (3mm);    
    
            \draw[semithick] (-4,-1.5) -- (3,-1.5) -- (3,3) -- (-4,3) -- (-4,-1.5);
            \draw[] (-3,-1) -- (2,-1) -- (2,2.5) -- (-3,2.5) -- (-3,-1);
    
                \draw (-4,3) node[anchor=north west] {\small \( B \)};
                \draw (-3,2.5) node[anchor=north west] {\small \( B' \)};
    
            \end{tikzpicture} 
            \caption{}\label{fig: c}
        \end{subfigure}

        \caption{In the three figures above we let the green and blue disks represent the connected components of \( \mathcal{G}(\omega,0) \) and \( \mathcal{G}(0,\omega') \) respectively for some \( \omega \in \Sigma_{P_B} \) and \( \omega' \in \Sigma_{P_{B'}} \). The clusters of partially overlapping disks correspond to connected components in \( \mathcal{G}(\omega,\omega') \). In Figure~\ref{fig: b} the red area corresponds to the set \( \smash{\hat P} \), and in Figure~\ref{fig: c} the red area corresponds to the set \( \smash{\hat{\hat P}} \).}
        \label{figure: disks}
    \end{figure}
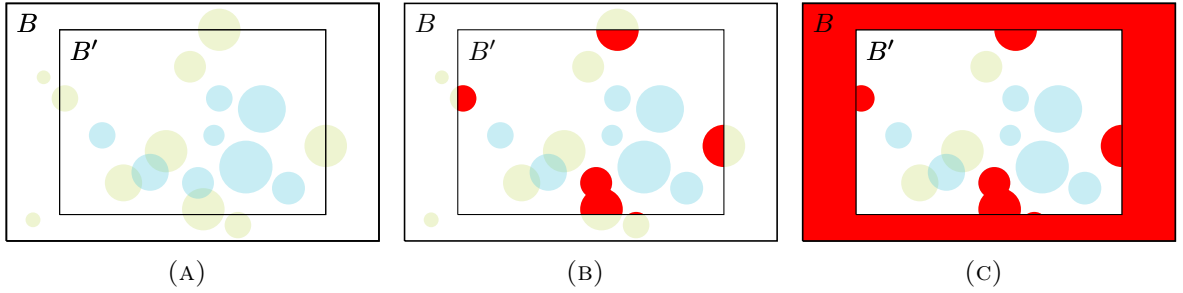%
    Note also that the set \( \hat P \) can be determined without looking outside \( \hat {\hat P} \) in \( \omega \) and \( \omega' \). Together, these observations imply that if we let \( \omega \sim \mu_{B,\beta} \) and \( \omega' \sim \mu_{B',\beta} \) be independent, then, conditioned on \( \hat P = \hat P(\omega,\omega')\), we have
    \begin{equation}\label{eq: coupling req old}
         \omega |_{P_{B'} \smallsetminus \hat P} 
         \overset{d}{=}  
         \omega'|_{P_{B'} \smallsetminus \hat P}.
    \end{equation}
    We now use this observation to construct a coupling  \( (\hat \omega,\hat \omega') \) of \( \hat \omega \sim \mu_{B,\beta} \) and \( \hat \omega' \sim \mu_{B',\beta} \).
    To this end, let \( \omega \sim \mu_{B,\beta} \) and \( \omega' \sim \mu_{B',\beta} \) be independent.
    Given \( \hat P = \hat P(\omega,\omega') \), define 
    \begin{equation*}
        \begin{cases} 
            \hat \omega \coloneqq  \omega\cr
            \hat \omega' \coloneqq \omega'|_{\hat P} + \omega|_{P_{B'}\smallsetminus \hat P}.
        \end{cases}
    \end{equation*}
    By~\eqref{eq: coupling req old}, we have  
    \begin{equation}\label{eq: coupling req new}
        \hat \omega |_{P_{B'} \smallsetminus \hat P} \overset{d}{=} \hat \omega'|_{P_{B'} \smallsetminus \hat P},
    \end{equation}
    and hence \( \hat \omega' \sim \mu_{B',\beta} \). Since \( \hat \omega \sim \mu_{B,\beta} \) by definition, it follows that \( (\hat \omega, \hat \omega') \) is a coupling of \( \hat \omega \sim \mu_{B,\beta} \) and \( \hat \omega' \sim \mu_{B',\beta} \).
    By the coupling characterization of total variation distance, we therefore have
    \begin{equation*}
        \begin{split}
            &\dist_{TV} ( \omega|_{P_0},  \omega'|_{P_0}) 
            \leq 
            \mu_{B,\beta}\times \mu_{B',\beta}\Bigl( \bigl\{ \omega \in \Sigma_{P_B} ,\, \omega' \in \Sigma_{B',\beta}\colon
            \hat \omega |_{P_0} \neq \hat \omega'|_{P_0}\bigr\} \Bigr).
        \end{split}
    \end{equation*}
    If \( \hat \omega |_{P_0} \neq \hat \omega'|_{P_0}  \), then
    \begin{equation*}
        P_0 \leftrightarrow \hat P_B \smallsetminus P_{B'},
    \end{equation*} 
    and hence it follows that
    \begin{equation*}
        \begin{split}
            &\dist_{TV} ( \omega |_{P_0},  \omega'|_{P_0})
            \leq  
            \mu_{B,\beta}\times \mu_{B,\beta}\Bigl( \bigl\{ \omega ,\omega' \in \Sigma_{P_B} \colon
            P_0 \leftrightarrow (P_B\smallsetminus P_{B'}) \text{ in }   \mathcal{G}(\omega ,\omega') \bigr\} \Bigr).
        \end{split}
    \end{equation*}
    This completes the proof.
\end{proof}

We now use the previous lemma to give a proof of Theorem~\ref{theorem: chatterjees correlation result}.

\begin{proof}[Proof of Theorem~\ref{theorem: chatterjees correlation result}]
    By Lemma~\ref{lemma: total variation bound}, we have
    \begin{equation*}
        \begin{split}
            &\dist_{TV} ( \omega|_{P_0},  \omega'|_{P_0})  
            \leq  
            \mu_{B,\beta}\times \mu_{B',\beta}\Bigl( \bigl\{ \omega \in \Sigma_{P_B}, \omega' \in \Sigma_{P_{B'}} \colon
             P_0 \leftrightarrow (P_B\smallsetminus P_{B'}) \text{ in }   (\omega,\omega')) \bigr\} \Bigr).
        \end{split}
    \end{equation*}
    On the other hand, by Proposition~\ref{proposition: probability of connected vs probability of path}, we have
    \begin{equation*}
        \begin{split}
            &
            \mu_{B,\beta}\times \mu_{B',\beta}\Bigl( \bigl\{ \omega \in \Sigma_{P_B}, \omega' \in \Sigma_{P_{B'}} \colon
             P_0 \leftrightarrow (P_B\smallsetminus P_{B'})  \text{ in }   (\omega,\omega') \bigr\} \Bigr) 
            \\&\qquad \leq 
            C_1 |\P_0|  \bigl( C_2 \alpha(\beta) \bigr)^{ \dist_{B,B'}(P_0,P_B\smallsetminus P_{B'}) }.
        \end{split}
    \end{equation*}
    Combining the previous equations, we thus obtain
    \begin{align*}
        &\dist_{TV} ( \omega|_{P_0},  \omega'|_{P_0})   
        \leq 
         C_1  | P_0| \bigl( C_2 \alpha(\beta) \bigr)^{ \dist_{B,B'}(P_0,P_B\smallsetminus P_{B'}) } 
    \end{align*}
    as desired.
\end{proof}

\section{The expected spin at a plaquette}\label{sec: expected spin}

In this section, we give proofs of Theorem~\ref{theorem: typical spin at plaquette} and Theorem~\ref{theorem: correlation length ii}.

\begin{proof}[Proof of Theorem~\ref{theorem: typical spin at plaquette}]
    We first define two useful events.
    Let 
    \begin{equation*}
        A \coloneqq \bigl\{ \sigma \in \Sigma_{E_B} \colon \exists \text{ irreducible } \nu \in \Sigma_{P_B} \text{ with }  (d\sigma)|_{\support \nu} = \nu,\,  p \in \support \nu \text{ and } |(\support \nu)^+|\geq 11  \bigr\},
    \end{equation*}
    and, for \( e \in \partial p \) and \( g \in G  \), let 
    \begin{equation*}
        A_{e,g} \coloneqq \bigl\{ \sigma \in \Sigma_{E_B} \colon \forall  p' \in  \partial e  \text{ we have } (d\sigma)_{ p'} = g 
        \bigr\}.
    \end{equation*}  
    Note that since \( \dist_{B} \bigl( \{p\}, \delta P_B \bigr) > 6 \), these events are all well-defined. Moreover, by Lemma~\ref{lemma: small vortices} and Lemma~\ref{lemma: minimal vortices}, since \( \dist_{B} \bigl( \{p\}, \delta P_B \bigr) > 11 \), we have
    \begin{equation}\label{eq: first event equality}
        \bigl\{ \sigma \in \Sigma_{E_N} \colon (d\sigma)_p \neq 0 \bigr\} = A \cup \bigcup_{e \in \partial p} \bigsqcup_{g \in G\smallsetminus \{ 0 \}} A_{e,g}
    \end{equation}

We now give upper bounds of the \( \mu_{B,\beta}\)-measure of the events defined above.
    To this end, note first that by Proposition~\ref{proposition: vortex probability}, applied with \( P = \{ p,-p \} \) and \( M = 11 \), we have
    \begin{equation}\label{eq: A}
        \mu_{B,\beta}( A) \leq \frac{5^{10} \alpha(\beta)^{11}}{1 - 5\alpha(\beta) } .
    \end{equation}
    If for some \( e \in \partial p \) the event \( A_{e,0}^c \) happen, then there must exist an irreducible  plaquette configuration \( \nu \in \Sigma_{P_B}\) with \( (d\sigma)|_{\support \nu} = \nu \) and \( \support \nu \cap \hat \partial e \neq \emptyset \). From Proposition~\ref{proposition: vortex probability}, applied with \( p' \in \hat \partial e \) and \( M= 6 \),  and a union bound, we obtain
    \begin{equation}\label{eq: Ae0}
        0 \leq 1 - \mu_{B,\beta}(A_{e,0}) \leq \sum_{p' \in \hat \partial e} \frac{5^{5} \alpha(\beta)^{6}}{1 - 5\alpha(\beta) }.
    \end{equation}
    Next, note that if  \( \{ e,e' \} \subseteq \partial p \) and \( g \in G\smallsetminus \{ 0 \} \) then it follows from Proposition~\ref{proposition: vortex probability}, applied with \( P = (\hat \partial e \cup \hat \partial (-e)) \cup (\hat \partial e' \cup \hat \partial (-e'))  \) and \( M = |P^+| = 6+6-1 = 11 \), that
    \begin{equation}\label{eq: Aeeg}
        \sum_{g \in G\smallsetminus \{ 0 \}} \mu_{B,\beta}(A_{e,g} \cap A_{e',g}) \leq
         \frac{ \alpha(\beta)^{11}}{1 - 5\alpha(\beta) } . 
    \end{equation}

    Finally, if we for \( e = dx_j \in \partial p \) and \( g \in G \smallsetminus \{ 0 \} \) define \( \nu^{(e,g)} \) by
    \begin{equation*}
        \nu^{(e,g)}_{e'} \coloneqq d(g \, dx_j)
    \end{equation*}
    then \( d(\nu^{(e,g)})=0 \), \( \support  \nu^{(e,g)}  = \pm \hat \partial e \), and by Lemma~\ref{lemma: minimal vortices}, we have
    \begin{equation*}
        \sigma \in A_{e,g} \Leftrightarrow (d\sigma)|_{\pm \hat \partial e} = \nu^{(e,g)}.
    \end{equation*}
    Moreover, we have
    \begin{equation*}
        \sigma \in A_{e,0} \Leftrightarrow (d\sigma)|_{\pm \hat \partial e} = 0.
    \end{equation*}
    Consequently, we can apply Lemma~\ref{lemma: ratio} with \( \nu = \nu^{(e,g)} \) to obtain
    \begin{equation}\label{eq: g comparison}
        \frac{\mu_{B,\beta}(A_{e,g})}{\mu_{B,\beta}(A_{e,0})} = \frac{\mu_{B,\beta}\bigl( \{ \sigma \in \Sigma_{E_B} \colon (d\sigma)|_{\support \nu^{(e,g)}}=\nu^{(e,g)}\} \bigr)}{\mu_{B,\beta}\bigl( \{ \sigma \in \Sigma_{E_B} \colon (d\sigma)|_{\support \nu^{(e,g)}}=0 \} \bigr)} = \phi_\beta(g)^{12}
    \end{equation}

    We now combine the above equations to obtain~\eqref{eq: spin expectation}. To this end, note first that by the triangle inequality we have
    \begin{align}
        &\biggl| \mathbb{E}_{B,\beta}\Bigl[f\bigl((d\sigma)_p\bigr) \Bigr] - \Bigl( f(0) +
        \sum_{e \in \partial p} \sum_{g \in G} 
         \bigl(f(g)-f(0) \bigr) \, \phi(g)^{12} \Bigr) \biggr| \nonumber
         \\&\qquad 
         =
         \biggl| \mathbb{E}_{B,\beta}\bigl[f\bigl((d\sigma)_p\bigr)-f(0)\bigr] -  
        \sum_{e \in \partial p} \mu_{B,\beta} (A_{e,0})  \sum_{g \in G } 
         \bigl(f(g)-f(0) \bigr) \, \phi(g)^{12}  \nonumber
         \\&\qquad\qquad +  
         \sum_{e \in \partial p} \bigl( \mu_{B,\beta} (A_{e,0}) -1 \bigr) \sum_{g \in G} 
         \bigl(f(g)-f(0) \bigr) \, \phi(g)^{12} 
         \biggr|\nonumber
         \\&\qquad 
         \leq
         \biggl| \mathbb{E}_{B,\beta}[f((d\sigma)_p)-f(0)] -  
        \sum_{e \in \partial p} \mu_{B,\beta} (A_{e,0}) \sum_{g \in G } 
         \bigl(f(g)-f(0) \bigr) \, 
         \phi(g)^{12}   \biggr|\label{eq: first error term}
         \\&\qquad\qquad+ 
         \biggl| 
        \sum_{e \in \partial p} \bigl(1- \mu_{B,\beta} (A_{e,0})\bigr) \sum_{g \in G} 
         \bigl(f(g)-f(0) \bigr) \, 
         \phi(g)^{12}   \biggr|.\label{eq: second error term}
    \end{align}
    To get an upper bound on~\eqref{eq: first error term}, we first rewrite~\eqref{eq: first event equality} as a union of disjoint events;
    \begin{equation}\label{eq: second event equality}
        \{ \sigma \in \Sigma_{E_N} \colon (d\sigma)_p \neq 0 \} = A \sqcup \bigsqcup_{g \in G\smallsetminus \{ 0 \}} \biggl( \bigsqcup_{\substack{E \subseteq \partial p \mathrlap{\colon} \\ E \neq \emptyset}} (A^c  \bigcap_{e' \in E} A_{e',g} \bigcap_{e'' \in \partial p \smallsetminus E} A_{e'',g}^c) \biggr).
    \end{equation}
    Using first~\eqref{eq: g comparison} and then~\eqref{eq: second event equality}, we immediately obtain
    \begin{equation*}\label{eq: part i}
        \begin{split}
            &\Bigl| \mathbb{E}_{B,\beta}\bigl[ f\bigl((d\sigma)_p\bigr)-f(0) \bigr] -   \sum_{e \in \partial p} \mu_{B,\beta} (A_{e,0}) \sum_{g \in G }   \bigl( f(g) - f(0) \bigr) \, \phi_{B,\beta} (g)  \Bigr|
            \\&\qquad\overset{\eqref{eq: g comparison}}{=}
            \Bigl| \mathbb{E}_{B,\beta}\bigl[ f\bigl((d\sigma)_p\bigr)-f(0) \bigr] -  \sum_{e \in \partial p} \sum_{g \in G}   \bigl( f(g) - f(0) \bigr) \, \mu_{B,\beta} (A_{e,g})  \Bigr|
            \\&\qquad \overset{\eqref{eq: second event equality}}{\leq}
            \Bigl( \mu_{B,\beta}(A) + \sum_{\{e_1,e_2 \} \subseteq \partial p} \sum_{g \in G\smallsetminus \{ 0 \}}  \mu_{B,\beta} (A_{e_1,g} \cap A_{e_2,g})   \Bigr) \cdot \max_{g \in G} \bigl|f(g) -f(0)\bigr|   
            \\&\qquad \overset{\eqref{eq: A},\eqref{eq: Aeeg}}{\leq}
            \biggl( \frac{5^{10} \alpha(\beta)^{11}}{1 - 5\alpha(\beta) } +   \binom{|\partial p|}{2}  \frac{ \alpha(\beta)^{11}}{1 - 5\alpha(\beta) } \biggr)\cdot \max_{g \in G} \bigl|f(g) -f(0)\bigr|   
            .
        \end{split}
    \end{equation*}
    To obtain an upper bound on~\eqref{eq: second error term}, we use~\eqref{eq: Ae0} to get  
    \begin{equation*}\label{eq: part ii}
    \begin{split}
        &\biggl|\sum_{e \in \partial p} \bigl(1- \mu_{B,\beta} (A_{e,0})\bigr) \sum_{g \in G} 
         \bigl(f(g)-f(0) \bigr) \, 
         \phi(g)^{12}   \biggr|
         \\&\qquad\leq
         \sum_{e \in \partial p} \sum_{p' \in \hat \partial e} \frac{5^{5} \alpha(\beta)^{6}}{1 - 5\alpha(\beta) } \cdot \max_{g \in G} \bigl|f(g) -f(0)\bigr| \cdot \sum_{g \in G\smallsetminus \{ 0 \}} \phi_\beta(g)^{12} .
    \end{split}
    \end{equation*} 
    Finally, note that \( \binom{|\partial p|}{2} = 6 \),  that \( \sum_{e\in \partial p} |\hat \partial e| = 4 \cdot 6 \), and that
    \begin{equation*}
        \sum_{g \in G\smallsetminus \{ 0 \}} \phi_\beta(g)^{12} \leq \bigl(\sum_{g \in G\smallsetminus \{ 0 \}} \phi_\beta(g)^{2} \bigr)^6 = \alpha(\beta)^6 
    \end{equation*}
    Combining the above equations, we thus get
    \begin{equation*}
        \begin{split}
            &\biggl| \mathbb{E}_{B,\beta}\Bigl[f\bigl((d\sigma)_p\bigr) \Bigr] - \Bigl( f(0) +
        \sum_{e \in \partial p} \sum_{g \in G} 
         \bigl(f(g)-f(0) \bigr) \, \phi(g)^{12} \Bigr) \biggr|
         \\&\qquad\leq 
         \biggl( \frac{5^{10} \alpha(\beta)^{11}}{1 - 5\alpha(\beta) } +   
         \frac{ 6\alpha(\beta)^{11}}{1 - 5\alpha(\beta) }
         +
         \frac{24 \cdot 5^{5} \alpha(\beta)^{12}}{1 - 5\alpha(\beta) }  
        \biggr)\cdot \max_{g \in G} \bigl|f(g) -f(0)\bigr|.
        \end{split}
    \end{equation*}
    Since \( 5 \alpha(\beta)< 1 \), we have \( 5^{10} + 6 + 24 \cdot 5^{5} \alpha(\beta) < 5^{11}\), and hence the desired conclusion follows. 
\end{proof}

\begin{proof}[Proof of Theorem~\ref{theorem: correlation length ii}]
    Note first that, by definition, we have
    \begin{equation*}
        \mathbb{E}_{B,\beta}\bigl[f_1\bigl((d\sigma)_{p_1}\bigr)f_2\bigl((d\sigma)_{p_2}\bigl)\bigr] = \Cov \Bigl( f_1\bigl((d\sigma)_{p_1}\bigr),f_2\bigl((d\sigma)_{p_2}\bigr) \Bigr) + \mathbb{E}_{B,\beta}\bigl[\rho\bigl((d\sigma)_{p_1}\bigr)\bigr]\mathbb{E}_{B,\beta}\bigl[\rho\bigl((d\sigma)_{p_2}\bigr)\bigr].
    \end{equation*}
    Consequently, by the triangle inequality we have
    \begin{align*}
        &\biggl| \mathbb{E}_{B,\beta}\bigl[f_1\bigl((d\sigma)_{p_1}\bigr)f_2\bigl((d\sigma)_{p_2}\bigr)\bigr] - \prod_{j\in\{1,2 \}}\Bigl( f_j(0) + \sum_{e \in \partial p_j}  \sum_{g \in G} \bigl(f_j(g)-f_j(0) \bigr) \, \phi_\beta(g)^{12}  \Bigr) \biggr|
        \\&\qquad\leq
        \biggl| \Cov \Bigl( f_1\bigl((d\sigma)_{p_1}\bigr),f_2\bigl((d\sigma)_{p_2}\bigr) \Bigr) \biggr| 
        \\&\qquad\qquad + 
        \biggl| \mathbb{E}_{B,\beta}\bigl[f_1\bigl((d\sigma)_p\bigr) \bigr] \mathbb{E}_{B,\beta}\bigl[f_2\bigl((d\sigma)_{p'}\bigr)\bigr] - \prod_{j\in\{1,2 \}}\Bigl( f_j(0) + \sum_{e \in \partial p}  \sum_{g \in G} \bigl(f_j(g)-f_j(0) \bigr) \, \phi_\beta(g)^{12}  \Bigr) \biggr|.
    \end{align*}
    For the second term, note that for any \( x,y,a,b \in \mathbb{C} \), we have \( |xy-ab| \leq |(x-a)(y-b)| + |a(y-b)| + |b(x-a)| \). Using this inequality and Theorem~\ref{theorem: typical spin at plaquette}, we obtain
    \begin{align*}
        &
        \biggl| \mathbb{E}_{B,\beta}\bigl[f_1\bigl((d\sigma)_{p_1}\bigr) \bigr] \mathbb{E}_{B,\beta}\bigl[f_2\bigl((d\sigma)_{p_2}\bigr)\bigr] - \prod_{j\in\{1,2 \}}\Bigl( f_j(0) + \sum_{e \in \partial p_j}  \sum_{g \in G} \bigl(f_j(g)-f_j(0) \bigr) \, \phi_\beta(g)^{12}  \Bigr) \biggr|
        \\&\qquad\leq
        \prod_{j \in \{0,1 \}}  \frac{\bigl( 5 \alpha(\beta)\bigr)^{11}}{1 - 5\alpha(\beta) } 
            \cdot \max_{g \in G } \bigl|f_j(g) -f_j(0)\bigr|
            \\&\qquad\qquad +
            \sum_{j \in \{0,1 \}}  \frac{ \bigl( 5 \alpha(\beta)\bigr)^{11}}{1 - 5\alpha(\beta) } 
            \cdot \max_{g \in G } \bigl|f_j(g) -f_j(0)\bigr| \cdot \max_{g \in G } \bigl|f_{1-j}(g)\bigr|.
    \end{align*}
    Combining the previous equations and using Theorem~\ref{theorem: correlation length}, we get
    \begin{align*}
        &\biggl| \mathbb{E}_{B,\beta}\bigl[f_1\bigl((d\sigma)_{p_1}\bigr)f_2\bigl((d\sigma)_{p_2}\bigr)\bigr] - \prod_{j\in\{1,2 \}}\Bigl( f_j(0) + \sum_{e \in \partial p}  \sum_{g \in G} \bigl(f_j(g)-f_j(0) \bigr) \, \phi_\beta(g)^{12}  \Bigr) \biggr|
        \\&\qquad\leq
        C_1 \| f_1 \|_\infty \| f_2\|_\infty \bigl( C_2 \alpha(\beta)\bigr)^{ \dist_{B}(\{ p_1 \},\{ p_2 \}) } 
         +
         \biggl(  \frac{\bigl( 5 \alpha(\beta)\bigr)^{11}}{1 - 5\alpha(\beta) } \biggr)^2 \prod_{j \in \{0,1 \}} 
             \max_{g \in G } \bigl|f_j(g) -f_j(0)\bigr|
            \\&\qquad\qquad +
              \frac{\bigl( 5 \alpha(\beta)\bigr)^{11}}{1 - 5\alpha(\beta) } 
             \sum_{j \in \{0,1 \}} \max_{g \in G } \bigl|f_j(g) -f_j(0)\bigr| \cdot \max_{g \in G } \bigl|f_{1-j}(g)\bigr|.
    \end{align*}
    Noting that, for \( j \in \{ 0,1 \} \), we have
    \begin{equation*}
         \max_{g \in G } \bigl|f_j(g) -f_j(0)\bigr| \leq 2\max_{g \in G } \bigl|f_j(g)\bigr|,
    \end{equation*}
    we obtain~\eqref{eq: eq in last theorem} as desired.
\end{proof}

\medskip

\paragraph{\textbf{Acknowledgements.}}
The author acknowledges support from the Knut and Alice Wallenberg Foundation and from the Swedish Research Council, Grant No. 2015-05430. The author would also like to thank Fredrik Viklund and Jonatan Lenells for comments on the content of this manuscript.

\end{document}